\documentclass[reqno,11pt]{amsart}
\usepackage{mathtools,amsmath, amsfonts, amssymb, latexsym, amsthm,
  amscd, cancel, enumerate, rotating, comment, appendix, ulem,
  makecell, paralist,times,graphicx,pgf,dirtytalk,
appendix,float}

\usepackage{extarrows}
\usepackage{thmtools}
\usepackage[colorlinks=true, pdfstartview=FitV, linkcolor=blue, citecolor=blue, urlcolor=blue]{hyperref}
\usepackage[capitalise, noabbrev]{cleveref}
\usepackage{paralist}
\usepackage{pgf}
\usepackage{times}
\usepackage{enumitem}

\usepackage{caption}
\usepackage{subcaption}
\usepackage[T1]{fontenc}   
\usepackage{comment}
\usepackage{tikz}
\usetikzlibrary{graphs}

\definecolor{dblackcolor}{rgb}{0.0,0.0,0.0}
\definecolor{dbluecolor}{rgb}{0.01,0.02,0.7}
\definecolor{dgreencolor}{rgb}{0.2,0.4,0.0}
\definecolor{dgraycolor}{rgb}{0.30,0.3,0.30}

\newcommand{\st}{\,\colon \,}
\newcommand{\qforall}{\quad \mbox{for all} \quad}
\newcommand{\qand}{\quad \mbox{and} \quad}
\newcommand{\qor}{\quad \mbox{or} \quad}

\newcommand{\qfor}{\quad \mbox{for} \quad}
\newcommand{\qforsome}{\quad \mbox{for some} \quad}
\newcommand{\qwhere}{\quad \mbox{where} \quad}
\newcommand{\qwith}{\quad \mbox{with} \quad}

\newcommand{\G}{\mathcal{G}}

\newcommand\KK{\mathbb{K}}
\newcommand\rationals{\mathbb{Q}}

\newcommand{\rhk}{\widetilde{\mathrm{H}}}

\newcommand{\MinGen}{\mathrm{MinGens}}

\def\MC{\mathcal{C}}

\def\ML{\mathcal{L}}

\def\MP{{\mathcal P}}

\def\opn#1#2{\def#1{\operatorname{#2}}} % to make operators
\opn\Top{Top}
\opn\det{det}
\opn\ini{in}
\opn\height{height}
\opn\rank{rank}
\opn\supp{supp}
\opn\vsupp{vsupp}
\opn\gcd{gcd}
\opn\Min{Min}
\opn\sp{sp}
\opn\susp{susp}
\opn\lk{lk}
\opn\del{del}
\opn\LCM{LCM}
\opn\lcm{lcm}
\opn\reg{reg}
\opn\projdim{projdim}
\opn\Tor{Tor}
\opn\Nerve{Nerve}
\def\field{\KK}
\def\rk{\mathrm{rk}}

\let\iso = \cong
\newcommand{\bx}{\mathbf{x}} 
\newcommand{\by}{\mathbf{y}} 
\newcommand{\bz}{\mathbf{z}} 
\newcommand{\bm}{\mathbf{m}} 
\newcommand{\m}{\mathbf{m}} 
\newcommand{\n}{\mathbf{n}} 
\newcommand{\bn}{\mathbf{n}} 
\newcommand{\bu}{\mathbf{u}}

\newcommand{\facets}{\mathcal{M}}
\makeatletter
\@addtoreset{equation}{section}
\def\theequation{\thesection.\@arabic \c@equation}
\makeatother

\theoremstyle{plain}
\newtheorem{theorem}{Theorem}[section]
\newtheorem{lemma}[theorem]{Lemma}
\newtheorem{proposition}[theorem]{Proposition}
\newtheorem{corollary}[theorem]{Corollary}
\theoremstyle{definition}
\newtheorem{definition}[theorem]{Definition}

\newtheorem{example}[theorem]{Example}

\newtheorem{question}[theorem]{Question}
\newtheorem{running example}[theorem]{Running Example}

\title {Pairs of lattice complements with complementary homology}

\author[S. Faridi]{Sara Faridi}
\address[S. Faridi]
{Department of Mathematics \& Statistics,
Dalhousie University,
6297 Castine Way,
PO BOX 15000,
Halifax, NS,
Canada B3H 4R2
}
\email{faridi@dal.ca}

\author[D. Veer]{Dharm Veer}
\address[D. Veer]
{Department of Mathematics, 
Purdue University, 
150 N. University Street, 
West Lafayette, 
IN 47907-2067, USA
}
\email{dveer@purdue.edu}

\author[V. Welker]{Volkmar Welker}
\address[V. Welker]
{Philipps-Universit\"at Marburg,
Fachbereich Mathematik und Informatik,
35032 Marburg, Germany
}
\email{welker@mathematik.uni-marburg.de}

\thanks{Faridi's research is supported by NSERC Discovery Grant 2023-05929.}



\begin{document}

\begin{abstract} 
In this paper we study whether there is a global to local principle for the homology of order complexes of finite lattices. 
We explore whether the non-vanishing of the homology of a lattice in degree $a+b-2$ forces the existence of a pair of lattice-complements
whose open intervals below have non-vanishing homology in degrees $a$ and $b$ respectively. We consider the same question also in the language of multigraded 
free resolutions. We provide positive evidence for the existence of such pairs and give counterexamples to strengthenings of the question. 
\end{abstract}

\maketitle

\section{Introduction}

The study of order complexes of finite lattices and their homology 
goes back to work of Folkman \cite{Fol66}. Over the last 60 years it has led 
to numerous fruitful interactions of combinatorics with a wide range of 
mathematical fields ranging from arrangements of subspaces to commutative algebra. 
While new applications continue to arise, 
basic structural properties of lattice homology remain mysterious. In the 1960's Crapo~\cite{Crapo}, building on earlier work by Rota~\cite{Rot64}, discovered that the Euler-characteristic of a lattice  is influenced by the combinatorics of {\it complements}, where two elements of a lattice are called complements if their meet is the bottom and their join is the top of the lattice. In the 1970's Baclawski~\cite[Corollary~6.3]{Bac77} showed that if a lattice  {\it non-complemented}  -- i.e. it has an element with no complement -- then it is acyclic. In the subsequent years a stronger correlation between complements and the topology of lattices was discovered, shifting the focus to the homotopy type of the lattice. It was shown that a non-complemented lattice is contractible~(\cite[Theorem~3.3]{Bjo81}). Bj\"orner and Walker~\cite[Theorem 3.2]{BW83complements} showed more generally, that if one removes the complements of 
a fixed element from the proper part of a lattice, then the corresponding order complex is
contractible. By basic topology, this fact implies that homology and homotopy type of lattices 
are encoded in the local
behavior around the set of lattice complements of any fixed element 
(see \cite{BW83complements, Bjorner98complements}
for results exploiting this fact).  

In this work we study which consequences on the local structure are implied by
the non-vanishing of global homology. 
Motivated by questions on free resolutions of monomial ideals,
we ask if global non-vanishing homology forces local non-vanishing
homology of pairs of lattice complements in complementary homological dimensions.

 The central quest of the current paper is \cref{que:main} below, which is the lattice theoretic version of  the general question above. \cref{que:main} was inspired by conjectures on degrees of syzygies of monomial ideals in polynomial rings, and the fact that these syzygies can be calculated via lattice homology.

\begin{question}[\cite{Faridilatticecomplements}] \label{que:main}
Let $L$ be a lattice, $\field$ a field, and suppose that
\begin{itemize}
\item $\rhk_k\big(\,\overline{L};\KK\,\big) \neq 0$ for some $k \geq 0$,
\item $a,b \geq -1$ and $a+b = k-2$.
\end{itemize}
Are  there complements $x$ and $y$ in $L$ such that
$$
\rhk_{a}\big(\,(\hat{0},x)_L;\KK\,\big) \neq 0
\qand 
\rhk_{b}\big(\,(\hat{0},y)_L;\KK\,\big) \neq 0?
$$
\end{question}

Evidence for a positive solution to this question is provided in 
\cite{Faridilatticecomplements,FM22breakinghomology}, where a positive
answer is established for large classes of
lattices. In \cite{synor24subadditivity} a positive answer is given for
a weak version of the question, where the element $y$ only has to be a
{\it supplement} of $x$, meaning that  the condition $x \wedge y = \hat{0}$ is not required. 
A positive answer to \cref{que:main} would reveal a deep combinatorial
property of lattices and would have strong implications in commutative algebra
and the theory of subspace arrangements.

In this work we add more positive evidence to \cref{que:main}. Using the notation $\MC(x)$ to denote the  set of complements of an element $x \in L$, we show:

\begin{theorem}\label{thm:complement}
Let $L$ be a lattice such that $\rhk_{k}(\overline{L}; \field) \neq 0$ for some $k\geq 0$. Then, for every $x\in \overline{L}$, there exists  
a complement $y \in \MC(x)$  and $z \in (\hat{0},x]_L\cap \MC(y)$ such that
$$
\rhk_{b}\big(\,(\hat{0},y)_L; \field\,\big) \neq 0
\qand 
\rhk_{a}\big(\,(\hat{0},z )_L; \field\,\big) \neq 0
$$ 
for some $a,b\geq -1$ with $a+b = k-2$.

Moreover, if $\MC(x)$ is an antichain, then for every  $b'$ such that 
$\rhk_{b'}\big(\,(\hat{0},y )_L; \field\,\big) \neq 0$ we have
$$
\rhk_{a+b'+2}\big(\,\overline{L}; \field\,\big) \neq 0\,.
$$
\end{theorem}

A second result pointing towards a positive answer considers the case
when $L$ is the face lattice of a simplicial complex; that is the set of faces
of the simplicial complex ordered by inclusion together with a
maximal element $\hat{1}$.

\begin{theorem} \label{thm:simplicialcomplex}
    Let $L = \ML(\Delta)$ be the face lattice of a simplicial complex $\Delta$. If $\rhk_{k}(\overline{L};\field) \neq 0$ (equivalently, if $\rhk_{k}(\Delta;\field) \neq 0$), then for all $a,b\geq -1$ with $a+b = k-2$, there exist $\sigma, \tau\in \overline{L}$ such that 
    \begin{enumerate}
        \item $\sigma \cap \tau = \emptyset$,
        \item $\sigma \cup \tau \not\in \Delta$, and
        \item $\dim(\sigma) = a+1$, $\dim(\tau) = b+1$. 
    \end{enumerate}
    In particular,
    \[
\tau\in \MC(\sigma), \quad \rhk_{a}\big(\,(\hat{0},\sigma )_L; \field\,\big) \neq 0  \qand \rhk_{b}\big(\,(\hat{0},\tau )_L; \field\,\big) \neq 0.
\]
\end{theorem}

We call faces $\sigma$ and $\tau$ which satisfy properties~(1) and~(2)
in \cref{thm:simplicialcomplex} {\it complementary faces}. This leads to the definition of a {\it complemented} simplicial complex. It is then a simple consequence of the definition that, analogous to  Bj\"orner and Walker~\cite{BW83complements}, if a simplicial complex is not complemented then it is contractible (\cref{thm:noncomplemented-contractible}). 

In commutative algebra, the Betti
numbers of the minimal free resolution of a monomial ideal can be expressed in
terms of the homology groups of lower intervals in the so-called LCM lattice
(see \cite{GPW99}). These implications are discussed in detail in
\cite{Faridilatticecomplements,FM22breakinghomology}, and in \cite{synor24subadditivity} the positive answer to the weak version of \cref{que:main} was used to provide a positive answer to a 
question for monomial ideals. 
 In this context, \cref{thm:complement} translates into the following statement about multigraded Betti numbers of monomial ideals.

\begin{theorem}[See \cref{c:Betti-complemented}] 
Let $I$ be a monomial ideal in a polynomial ring, and let 
$\bm$ be a monomial such that  $\beta_{i,\bm}(S/I)\neq 0$ for some $i\geq 2$. 
 Then for every nontrivial monomial $\bx$ in the lcm  lattice of $I$ which strictly divides $\bm$ there are monomials $\by,\bz$ in the lcm lattice of $I$, and positive integers $a,b$, such that 
$$ 
\beta_{b,\by}(S/I) \neq 0, 
\quad 
\beta_{a,\bz}(S/I) \neq 0, 
\qand a+b=i\,,
$$ 
and moreover 
$$
\bz \mid \bx, \quad 
\gcd(\by,\bx) \not\in I, 
\ 
\gcd(\by,\bz) \notin I, 
\qand 
\lcm(\by,\bx)=\lcm(\by,\bz)=\bm\,.
$$
\end{theorem}

The Goresky--MacPherson formula (see \cite[Theorem A, p.~238]{GMP88})
provides a very similar expression for the cohomology of the complement of an
arrangement of linear subspaces in terms of the homology of lower intervals in
the intersection lattice of the arrangement. We do not discuss the implications
of \cref{que:main} and our results in this setting here. Nevertheless, we note
that a positive answer to \cref{que:main} would show that the existence of a
\say{global} cohomology class for the complement of an arrangement of linear
subspaces forces the existence of two local classes at complementary subspaces
and in complementary homology degrees.

The paper is organized as follows.  In \cref{s:prelim} basic definitions
and tools from lattice theory and poset topology are reviewed. We study the meet lattice, and 
recall in \cref{p:equivalences} a number of equivalent formulations of
\cref{que:main} from \cite{Faridilatticecomplements,FM22breakinghomology}.
In \cref{sec:3} we prove \cref{thm:complement}. As a preparation we exhibit in \cref{lem:homotopic} a surprising 
homotopy equivalence between lattice intervals with complements removed.
In \cref{s:faces} we show that in the setting of face lattices of simplicial complexes \cref{que:main} has a positive answer by proving \cref{thm:simplicialcomplex}. 
In \cref{sec:meet} we explore the question whether the Cohen-Macaulayness of a lattice $L$ is inherited by the meet-lattice $L_\wedge$. In this context, we give an affirmative answer for face lattices and a counterexample for a general lattice.

In \cref{s:Betti} we then discuss the implications 
of the results from \cref{sec:3} in commutative algebra when $L$ is the
LCM lattice of a monomial ideal.
In \cref{sec:discussion} we discuss further directions and open problems.

%%%%%%%%%%%%%%%%%%%%%%%%
\section{Preliminaries and variations of \cref{que:main}  }\label{s:prelim}
%%%%%%%%%%%%%%%%%%%%%%%%

We first introduce the basic concepts and results used in this paper.

An \textit{(abstract) simplicial complex} $\Delta$ on \textit{vertex set} $V(\Delta)$ is a (finite) non-empty collection of subsets of $V(\Delta)$ such that whenever $\sigma' \in \Delta$ and $\sigma \subseteq \sigma'$, then $\sigma \in \Delta$. The elements of $\Delta$ are called \textit{faces}. The maximal faces of $\Delta$ with respect to inclusion are called \textit{facets}.
For subsets $\sigma_1,\ldots, \sigma_r \subseteq V$ of a finite set $V$ we
write $\langle\, \sigma_1,\ldots, \sigma_r \,\rangle$ for the (inclusionwise)
smallest simplicial
complex containing $\sigma_1,\ldots, \sigma_r$ as faces and call $\{ \sigma_1,\ldots, \sigma_r\}$ a 
\textit{generating set} of $\Delta$.

For any $\sigma \in \Delta$, the \textit{dimension} of $\sigma$, denoted by $\dim(\sigma)$, is one less than the cardinality of $\sigma$.  
A \textit{vertex} of $\Delta$ is a face of dimension $0$. The dimension of $\Delta$ is given by $\max\{\dim(\sigma) : \sigma \in \Delta\}$.  

To each face $\sigma \in \Delta$ one can associate the following simplicial complexes:
\begin{itemize}
    \item the \textit{ link} of $\sigma$ in $\Delta$ is $\lk_\Delta(\sigma) := \{\tau \in \Delta : \sigma \cap \tau = \emptyset \text{ and } \sigma \cup \tau \in \Delta\}$;
    \item the \textit{ deletion} of $\sigma$ in $\Delta$ is $\del_\Delta(\sigma) := \{\tau \in \Delta : \sigma \cap \tau = \emptyset\}$.
\end{itemize}

For two simplicial complexes $\Delta_1$ and $\Delta_2$ on disjoint ground sets
their \textit{ join} $\Delta_1 * \Delta_2$ is the simplicial complex on
ground set $V(\Delta_1) \cup V(\Delta_2)$ whose faces are all unions $\sigma_1 \cup \sigma_2$ for $\sigma_1 \in \Delta_1$ 
and $\sigma_2 \in \Delta_2$.  
For distinct vertices $v, w \notin V(\Delta)$, the \textit{ suspension} of $\Delta$, denoted by $\susp(\Delta)$, is defined as
\[
\susp(\Delta) = \Delta * \langle \{v\}, \{w\} \rangle.
\]

Let $P$ be a partially ordered set, poset for short, with order relation $\leq$. 
Throughout the paper, all posets will be finite. For $x \in P$ we write $P_{\leq x}$ for
the subposet $\{ y \in P~\st~y \leq x\}$. Analogously defined are $P_{< x}$, $P_{\geq x}$ and
$P_{> x}$. 
If $x \leq y$ in $P$, the \textit{open interval} $(x,y)_P$ is the subposet $P_{> x} \cap P_{< y}$.  Again, analogously defined are the \textit{ half-closed intervals} $(x,y]_P$
and $[x,y)_P$. 
For $x,y \in  P$, we say $y$ \textit{covers} $x$ if $x<y$ and there is no element $z\in P$ such that $x<z<y$.
A poset is bounded if it contains a unique minimal element $\hat{0}$
and a unique maximal element $\hat{1}$.
A poset $P$ is a \textit{lattice} if any two elements $x$ and $y$ of $P$ have a greatest lower bound or \say{meet} denoted by $x \wedge y$, and a smallest upper bound or \say{join} denoted by $x \vee y$. Since all our posets are finite, any lattice is a bounded poset.
For a bounded poset $P$ we write $\overline{P}$ for $P \setminus \{\hat{0},\hat{1}\}$. Two elements $x,y \in L$ in a lattice $L$ are called
complements if $x \vee y = \hat{1}$ and $x \wedge y = \hat{0}$.
The lattice $L$ is said to be \textit{ complemented} if every element $x\in L$ has a complement. 
For $x\in L$, let $\MC(x)$ denote the set of complements of $x$ in $L$. 

A \textit{ chain} $C$ of a poset $P$ is a totally ordered subset of $P$.  The \textit{order complex} of $P$ is a simplicial complex on vertex set $P$ whose faces are chains of $P$. 
In this article, we often identify a poset with its order complex. In particular, we talk about the homology groups of a poset and its homotopy type. 

We will study poset topology using the classical Quillen Fiber lemma.
Recall that a map of posets $f:P \to Q$ is \textit{ order preserving } if
$$
x \leq y \quad \mbox{ implies  } \quad 
f(x) \leq f(y) 
\qforall
x,y \in P\,.
$$

\begin{theorem}[Quillen Fiber Theorem A]\label{lem:quillen}\cite[Theorem~A]{Quillen73theoremA} \cite[Proposition\ 1.6]{Quillen78fibertherem}
    Let $P$ and $Q$ be two posets and $f: P\to Q$ be an  order preserving map of posets. If $f^{-1}(Q_{\leq y})$ is contractible for all $y\in Q$, then $f$ induces a homotopy equivalence  $P\simeq Q$.
\end{theorem}

As an immediate consequence, one gets the following well known fact about closure operators (see for example \cite[Corollary 10.12]{Bjornertopologicalmethods}) whose proof we
add for convenience.
Recall that a \textit{closure operator} is  an order preserving map of posets  
$$f : P \rightarrow P \quad \mbox{such that} \quad
f(p) \geq p,  \qand 
f(f(p)) = f(p) \qforall 
p \in P.$$

\begin{corollary} \label{lem:closure}
Let $f : P \rightarrow P$ be a closure operator, then $P$ and $f(P)$ are homotopy equivalent.
\end{corollary}

\begin{proof}  
Let $y \in f(P)$. Then by $f(p) \geq p$ we have that
$f^ {-1}\big(\,f(P)_{\leq y}\,\big)$ has $y$ as its unique maximal element. Hence the order complex is a cone and the statement  follows from \cref{lem:quillen}.  
\end{proof}

Let $L$ be a lattice and $\facets (L)$ be the set of maximal elements of $\overline{L}$. 
We call  the set 

\begin{equation}\label{eq:Lcap}
L_\wedge = \big\{\, \bigwedge_{x \in N} x \neq \hat{0}~\st~\emptyset \neq N \subseteq \facets (L)\,\big\}
\cup \{\hat{0},\hat {1}\}
\end{equation}
the  \textit{meet-lattice} of $L$. 
The definition immediately implies that $L_\wedge$ is a lattice, with meet and join defined, for $x,y \in L_\wedge$, as follows:
\begin{equation}\label{eq:meet-join-meet}
    x \wedge_{L_\wedge}  y =  x \wedge_L y
     \qand 
   x \vee_{L_\wedge}  y =    
   \bigwedge_{ \substack{z\in \facets(L)\\ x, y \leq z}} z\,.
       \end{equation}
       In particular, $x \vee_{L_\wedge}  y = x \vee_{L}  y$ if  $x \vee_{L}  y \in L_\wedge$. It is clear from the definition that $\hat{1}_{L_\wedge}=\hat{1}_L$. 
The fact that $\overline{L_\wedge} = \big\{\, \bigwedge_{x \in N} x \neq \hat{0}~\st~\emptyset \neq N \subseteq \facets (L)\,\big\}$
also guarantees that $\hat{0}_{L_\wedge}=\hat{0}_L$.

\begin{proposition} \label{lem:intlattice}
    Let $L$ be a lattice and $x \in L$.
   \begin{enumerate} 
    \item      The poset
     $ (\overline{L_\wedge})_{> x} = \overline{(L_{\geq x})_\wedge}$ 
            is homotopy equivalent to $\overline{L}_{> x}$;
    \item If $x \not\in L_\wedge$ then $\overline{L}_{> x}$ is contractible; 
    \item The posets $\overline{L_\wedge}$ and $\overline{L}$ are homotopy equivalent;
    \item If $L_\wedge$ is complemented then $L$ is complemented;
    \item If $\overline{L}$ is not contractible, then both $L$ and $L_\wedge$ are complemented.
\end{enumerate}
\end{proposition}

\begin{proof}
For~(1) and (2), take $x \in L$. Observe that the minimal element of $L_{\geq x}$ is $x$ and that 
$$
\{\,z\in \facets(L) \st x \leq z\} 
= 
\facets(L_{\geq x})\,.
$$ 
Then 
$$
\overline{(L_{\geq x})_\wedge}
= 
\big\{\, \bigwedge_{z \in N} z \neq x \st \emptyset \neq N \subseteq \facets (L_{\geq x})\,\big\}
=
\{ y  \in \overline{L_\wedge} \st x < y\}
=
(\overline{L_\wedge})_{> x}\,.
$$ 
We can define the map  of posets
$f:\overline{L}_{>x} \to \overline{L}_{>x}$
where 
$$f(z)= \bigwedge_{\genfrac{}{}{0pt}{}{y \in \facets (L_{\geq x})}{z \leq y}} y 
=
\bigwedge_{\genfrac{}{}{0pt}{}{y \in \facets (L)}{x <z \leq y}} y
\qfor z \in \overline{L}_{>x}\,.
$$
Then it is easily verified that $f$ is a closure operator, and therefore by \cref{lem:closure} $\overline{L}_{>x}$ is homotopy 
equivalent to $f(\overline{L}_{>x})=\overline{(L_{\geq x})_\wedge}$.
If $x \not\in L_{\wedge}$ the latter has $\bigwedge_{z \in \facets(L_{\geq x})} z$ as its unique minimal element and hence is contractible.

If we set $x=\hat{0} \in L$, then  $\overline{(L_{\geq x})_\wedge} = \overline{L_\wedge}$ and $\overline{L}_{>x} = \overline{L}$ are homotopy equivalent. This settles~(3).

For~(4) assume $L$ is not complemented. Then there is $x \in L$
    such that $x$ has no complement in $L$. Let
    $x_\wedge \in L_\wedge$ be the meet $\bigwedge_{x \leq z \in \facets(L)} z$
    of all maximal elements of $\overline{L}$ above $x$. Then by assumption 
    $x_\wedge$ has a complement $y \in L_\wedge$ in $L_\wedge$. 
    
    Since $x$ is not complemented in $L$ the element $y$ cannot be a complement of
    $x$ in $L$.
    Thus either $x \vee y \neq \hat{1}$ or $x \wedge y \neq \hat{0}$. 
    If $x \vee y \neq \hat{1}$ then there is a
    maximal element $z \in \facets(L)$ with $x,y \leq z$. 
    But then 
    $x_\wedge \leq z$ and $x_\wedge \vee y \leq z$ contradicting the  that
    $y$ is a complement of $x_{\wedge}$ in $L_{\wedge}$. 
    It follows that  $x \wedge y \neq \hat{0}$.
    Then $u = \bigwedge_{x \wedge y \leq z \in \facets(L)} z$ satisfies 
    $u \in L_\wedge$, $u \neq \hat{0}$ and $u \leq x_\wedge,y$. This contradicts 
    the fact that $y$ is a complement of $x_{\wedge}$ in $L_\wedge$. 
    Thus the assumption is false and $L$ is complemented. 
    
    Statement~(5) follows from Statement~(3) and the Bj\"orner and Walker's result~\cite{BW83complements} that non-comple\-men\-ted lattices are contractible.
 \end{proof}

  The converse of \cref{lem:intlattice}~(4) is false even if $L$ is the face lattice of
  a simplicial complex, where $\wedge$ is replaced with $\cap$. Consider the simplicial complex
  $\Delta$ on ground set $\{1,2,3,4\}$ and maximal simplices $\{1,2\},\{2,3\},\{3,4\}$.
  Then $\ML(\Delta)$ is easily seen to be complemented. On the other hand
  $$\ML_\cap(\Delta)
  = \big\{\, \hat{0},\hat{1},\{2\},\{3\},\{1,2\},\{2,3\},\{3,4\}\,\big\},$$
  in which $\{2,3\}$ has no complement. However,  $\{1\}$ is a complement of
  $\{2,3\}$ in $\ML(\Delta)$.

 We now prepare for the reformulation of \cref{que:main} in terms of 
commutative algebra.
For a simplicial complex $\Delta$ a \textit{ non-face} is a subset $\rho \subseteq V(\Delta)$
such that $\rho \not\in \Delta$. A non-face $\rho$ is called \textit{minimal non-face} if all of its proper
subsets are faces of $\Delta$. Let $S_\Delta = \field[\,x_v \st v\in V(\Delta)\,]$  be a polynomial
ring  over the field $\field$, and for  a subset $\rho$ of $V(\Delta)$ let $x_\rho = \prod_{v \in \rho} x_v$ be a square-free monomial ideal in $S_\Delta$. 
The \textit{Stanley-Reisner ideal}  of $\Delta$ is the monomial ideal 
$$
I_\Delta=(x_\rho \st \rho 
\mbox{ minimal non-face of }
\Delta)\,.
$$
For $\rho \subseteq V(\Delta)$, let $\rho^c=V(\Delta) \setminus \rho$. Then 
the simplicial complex $\Delta^{\vee}=\{\rho^c \st \rho  \notin \Delta\}$ is called the \textit{Alexander Dual}
of $\Delta$. By definition, 
\begin{equation}\label{eq:alex-facets}
\rho^c \mbox{ facet of } \Delta^ \vee
\iff 
\rho \mbox{ minimal non-face of } \Delta 
\iff 
x_\rho \mbox{ minimal generator  of } I_\Delta 
\,.
\end{equation}
More generally, if  $\G$ is a set of squarefree monomials in $S_\Delta$, then  
\begin{equation}\label{eq:dual-gens}
\Delta^\vee=\langle \rho^c \st x_\rho \in \G \rangle
\iff 
I_\Delta=(\G)
\,.
\end{equation}

A (not necessarily squarefree) monomial ideal $I$ in $S = \field[x_1,\ldots, x_n]$
has a unique minimal monomial generating set
$\MinGen(I)$ (see, for example,~\cite[Proposition 1.1.6]{HHmonomialideals}). For an arbitrary monomial generating set $\G$ of $I$
 its \textit{LCM-lattice} is the set $\LCM(\G)=\big\{\lcm(\G')~\st~\G' \subseteq \G\,\big\}$
 ordered by divisibility, where we adopt the convention that
 $\lcm(\emptyset) = 1 \in \field$. It is easily seen that $\LCM(\G)$ is a lattice
 with unique minimal element $\hat{0} = 1 \in \field$ and unique maximal element
 $\hat{1} = \lcm(\G)$. We write $\LCM(I)$ for $\LCM(\G)$ when $\G = \MinGen(I)$ is the unique minimal monomial
 generating set of $I$. 
 The LCM-lattice can be used to determine the multigraded
 Betti numbers $\beta_{i,\m}(S/I)$ 
 of $S/I$ for a monomial ideal $I$ generated by a set of monomials $\G$ (\cite[Proposition 58.11]{Pee11}, see \Cref{s:Betti} for more details).

 \begin{equation}\label{eq:lcm2} 
  \beta_{i,\m}(S/I) = 
\begin{cases}
 \dim_\KK\big( \,\rhk_{i-2}(\,(\hat{0},\m)_{\LCM(\G)};\KK\,)\,\big) & \text{ for }
i \geq 1 \text{ and } \m \in \LCM(\G) \setminus \{\hat{0}\}\\
0 & \text{ otherwise}\,.
\end{cases}
\end{equation} 
We also write $\beta_{i,i+j}(S/I)$ for the sum of $\beta_{i,\m}$ over
all monomials $\m$ of degree $i+j-1$.

  For a monomial $\m = x_1^{a_1}\cdots x_n^ {a_n}$ we define its polarization as the monomial
 $$\MP(\m) = x_{1,1}\cdots x_{1,a_1} \cdots x_ {n,1}\cdots x_{n,a_n}$$
 in the polynomial ring containing the variables $x_{i,j}$ for $1 \leq i \leq n$ and
 $1 \leq j \leq a_i$. 
If $\G$ is a generating set of $I$ consider
 $\MP(\G) = \{\, \MP(\m)~\st~\m\in \G\,\}$. Then $\LCM(\G)$ and $\LCM(\MP(\G))$ are isomorphic. 
 If $\G$ in addition satisfies that $\m$ divides $\lcm(\,\MinGen(I)\,)$ for all $\m \in \G$, then 
 we call the squarefree monomial ideal $\MP(I)$
 the \textit{polarization} of $I$. Note that under the hypothesis that $\m$ divides $\lcm(\,\MinGen(I)\,)$ for all $\m \in \G$, all $\G$ lead to the same $\MP(I)$.

For a simplicial complex $\Delta$, its \textit{ face lattice} $\ML(\Delta)$ is the poset on 
$\Delta \cup \{\hat{1}\}$ with ordering by inclusion and $\hat{1}$ as the unique maximal
element. It is easily seen that $\ML(\Delta)$ is a lattice with
$\hat{0} = \emptyset$ as its unique minimal element. 
The order complex of $\overline{\ML(\Delta)}$ is the barycentric subdivision
of $\Delta$ and hence homeomorphic to $\Delta$.

For a nonempty set $\Gamma$ of subsets of a finite set, we define $\ML(\Gamma)$ to be the set $\Gamma \cup \{\hat{0}=\emptyset,\hat{1}\}$
considered as a subposet of $\ML(\langle \Gamma \rangle)$.

For a lattice $L$, let $\Gamma_L$ be a set of subsets of $\overline{L}$ defined as  
     $$
     \Gamma_L=\{L_{\leq x} \st x \in \overline{L}\} \,.
     $$
Then $\Gamma_L$ is a poset under inclusion, and we can define a lattice $\ML(\Gamma_L)$ on the poset $\Gamma_L \cup \{\hat{0},\hat{1}\}$ with join and meet operations defined as 
       $$
    L_{\leq x} \wedge_{\ML(\Gamma_L)} L_{\leq y} =  L_{\leq x \wedge y} = L_{\leq x} \cap L_{\leq y} 
     \qand 
    L_{\leq x} \vee_{\ML(\Gamma_L)} L_{\leq y} =    L_{\leq x \vee y} \,.
    $$     
 The following lemma is a simple reformulation of
 the well known result that any lattice can be embedded as a meet-semilattice into a Boolean lattice. The formulation of the lemma is more suitable for our purpose to reformulate \cref{que:main} in  terms of simplicial complexes. 
 Recall that the \textit{order dual} of a poset $P$ ordered by $\leq$ is the poset $P^ *$ on the same set with order relation $\leq_*$ defined by
 $$
 p \leq_* q 
  \iff
 q \leq p\,.
 $$
 \begin{lemma} \label{lem:represent} \label{lem:lcm}
     Let $L$ be  a lattice, and   $\Delta_L$ be the simplicial complex generated by the set $\Gamma_L$. Then 
     \begin{enumerate} 
     \item $L\cong \ML(\Gamma_L)$; 
      \item $L \cong \LCM(\G)^*$ where $\G$ is a generating set of $I_{\Delta_L^\vee}$. 
 \end{enumerate}
 \end{lemma}
 
 \begin{proof}      
     Define $\phi:L \to \ML(\Gamma_L)$ by setting 
     $$
     \phi(x)=L_{\leq x} \qfor x \in \overline{L}, 
     \quad 
     \phi(\hat{0}) = \hat{0}, 
     \qand 
     \phi(\hat{1}) = \hat{1}\,.
     $$
     The map $\phi$ is clearly order preserving and bijective, and moreover $\phi^{-1}$ is also order preserving. Hence $\phi$ is a lattice isomorphism.

Now consider the set of monomials 
$\G = \{ \bx_{\rho} \st \rho^c \in \Gamma_L\,\}$ which is in clear
bijection to $\Gamma_L$. By \eqref{eq:dual-gens}, $\G$ is
a generating set of $I_{\Delta_L^ \vee}$. By construction, $\G$ ordered by divisibility is order
dual to $\Gamma_L$ ordered by set inclusion. In addition, the lcm of monomials 
$\bx_{\rho_1} ,\ldots, \bx_{\rho_r}$ 
in $\G^\vee$ is $x_{\rho}$ for $\rho = \bigcup_{i=1}^r \rho_i$, hence $\rho^c = \bigcap_{i=1}^r \rho_i^c$.
Since set intersection is the meet operation in $\ML(\Gamma_L)$ and taking lcms is the join 
operation in $\LCM(\G^\vee)$ the claim in~(2) follows.
\end{proof}

 The following proposition from \cite{FM22breakinghomology} is now an immediate consequence. In the statement below, for a monomial ideal $I$,  let $\lcm(I)$ denote the least common multiple of the unique minimal generating set of $I$.

 \begin{proposition}\label{p:equivalences}
     For integers $k\geq 0$ and $a,b\geq -1$, with $k-2=a+b$ the following statements are equivalent.
    \begin{itemize}
         \item[($L_{\leq}$)] For all lattices $L$ with $\rhk_{k} (\overline{L};\field) \neq 0$
         there are $x,y \in L$ and $x \vee y = 1$, $x \wedge y = \hat{0}$ such that
         $$\rhk_{a}\big(\,(\hat{0},x)_L;\field\,\big) \neq 0 \qand
          \rhk_{b}\big(\,(\hat{0},y)_L;\field\,\big) \neq 0.$$
           \item[($L_{\geq}$)] For all lattices $L$ with $\rhk_{k} (\overline{L};\field) \neq 0$
         there are $x,y \in L$ and $x \vee y = 1$, $x \wedge y = \hat{0}$ such that
         $$\rhk_{a}\big(\,(x,\hat{1})_L;\field\,\big) \neq 0 \qand 
          \rhk_{b}\big(\,(y,\hat{1})_L;\field\,\big) \neq 0.$$
              \item[($L_\wedge$)] For all lattices $L$ with $L = L_\wedge$ and
              $\rhk_{k} (\overline{L};\field) \neq 0$
         there are $x,y \in L$ and $x \vee y = 1$, $x \wedge y = \hat{0}$ such that
         $$\rhk_{a}\big(\,(x,\hat{1})_L;\field\,\big) \neq 0 \qand 
          \rhk_{b}\big(\,(y,\hat{1})_L;\field\,\big) \neq 0.$$
          
          \item[($\Delta$)] For all simplicial complexes $\Delta$ with
          $\rhk_{k} (\Delta;\field) \neq 0$ there are faces $\sigma,\tau \in \Delta$
          and $\sigma \cap \tau = \emptyset$ and $\sigma \cup \tau \not\in \Delta$
          for which 
          $$
          \rhk_{a}\big( \,\lk_\Delta(\sigma);\field\,\big) \neq 0
          \qand  
          \rhk_{b}\big(\,\lk_\Delta(\tau);\field\,\big) \neq 0.
          $$
          
          \item[($I$)] For all monomial ideals $I$  for which 
          $\rhk_{k}\big(\overline{\LCM(I)};\field\,\big)
          \neq 0$ there are 
          $$\m,\m' \in \LCM(I) \qwith 
          \lcm(\m,\m') =  \lcm(I) \qand 
          \gcd(\m,\m')\notin I
          $$ 
          such that
          $$\rhk_{a}\big(\,(\hat{0},\m)_{\LCM(I)};\field\,\big) \neq 0
          \qand  
          \rhk_{b}\big(\,(\hat{0},\m')_{\LCM(I)};\field\,\big) \neq 0;$$
          or, equivalently,
          $$\beta_{a+2,\m}(S/I)
          \neq 0
          \qand  
          \beta_{b+2,\m'}(S/I)
          \neq 0.$$  
         
     \end{itemize}
 \end{proposition}

 \begin{proof} The equivalence of ($L_{\leq }$) and ($L_{\geq}$) follows immediately from the
 fact that  the order dual of a lattice is a lattice, and the bijection between the (maximal) chains of the two posets leads to the fact the order complexes of a poset
 and its order dual are isomorphic.
 
 The statement of ($\Delta$) is ($L_{\geq}$) when applied to $L = \ML(\Delta)$. Note that 
 for $\sigma \in \Delta$ one has $(\sigma,\hat{1})_{\ML(\Delta)} = \overline{\ML(\lk_\Delta(\sigma)}$ 
 and that a simplicial complex is homeomorphic to the order complex of its face lattice.
 So ($L_{\geq}$) implies ($\Delta$).

 Assume ($\Delta$) holds. For a lattice $L$ there exist by \cref{lem:represent} a simplicial
 complex $\Delta$ and a generating set $\Gamma \subseteq \Delta$ such that
 $\ML(\Gamma)$ is isomorphic to $L$. Since the meet $\wedge$ in $\ML(\Gamma)$ is the set intersection 
 $\cap$ and since $\Gamma$ contains all facets of $\Delta$ it follows $\ML(\Delta)_\wedge
 = \ML(\Delta)_\cap \subseteq \ML(\Gamma)$ and $\ML(\Gamma)_\cap = \ML(\Delta)_\cap$. Moreover,
 $L_\wedge$ and $\ML(\Gamma)_\cap$ are isomorphic. 

 Using \cref{lem:intlattice} we conclude that $L, \ML(\Gamma)$ and $\ML(\Delta)$ have isomorphic homology.
 Thus by ($\Delta$) there are $\sigma,\tau \in \Delta$
 with  $\rhk_{a}\big( \,\lk_\Delta(\sigma);\field\,\big) \neq 0$
 and $\rhk_{b}\big(\,\lk_\Delta(\tau);\field\,\big) \neq 0$ such that 
 $\sigma \cap \tau = \emptyset$ and $\sigma \cup\tau \not\in \Delta$.
 By \cref{lem:intlattice} it
 follows that if $(\sigma,\hat{1})_{\ML(\Gamma)}$ has non-trivial homology then 
 $\sigma \in \ML(\Delta)_\cap = \ML(\Gamma)_\cap$ and $(\sigma,\hat{1})_{\ML(\Gamma)}$ and
 $\lk_\Delta(\sigma)$ have isomorphic homology. The same holds for $\tau$. 
 Since by \cref{lem:represent} the meet in $\ML(\Gamma)_\cap$ is set intersection we have
 $\sigma \wedge \tau = \hat{0}$ in $\ML(\Gamma)_\cap$. The fact that $\sigma \vee \tau= \hat{1}$ in
 $\ML(\Gamma)_\cap$ follows immediately from $\sigma \cup \tau \not \in \Delta$ and $\Gamma\subseteq \Delta$.
 This proves, ($\Delta$) implies ($L_{\wedge}$).

Assume ($L_\wedge$) holds. Let $L$ be an arbitrary lattice. Then by \cref{lem:intlattice} it follows that $L$ and $L_\wedge$ have
 isomorphic homology groups and that if $(x,\hat{1})_L$ or $(y,\hat{1})_L$ have non-trivial homology
 then $x$, $y \in L_\wedge$. For $x$, $y \in L_\wedge$ again  \cref{lem:intlattice} shows
 that $(x,\hat{1})_L$ and $(x,\hat{1})_{L_\wedge}$ have the isomorphic homology and
 $(y,\hat{1})_L$ and $(y,\hat{1})_{L_\wedge}$ have isomorphic homology.
 Since by construction $x \wedge y = \hat{0}$ and $x \vee y = \hat{1}$ in $L_\wedge$ imply the same in $L$
 the assertion of ($L_\geq$) follows. So $L_{\wedge}$ implies $L_{\geq}$.

  The equivalence of ($\Delta$) and ($I$) is a consequence of the fact that polarization preserves the lcm lattice formula~(\cite{GPW99}), and Eagon and Reiner's~\cite{ER98} formula for multigraded Betti numbers in terms of homology of links of simplicial complexes (see~\cite[Section~2.3.1]{FM22breakinghomology} for a detailed argument). 
\end{proof}

The equivalent homology statements in \cref{p:equivalences} inspire the definition of \textit{ complementary homologies} in each setting. In a nutshell: does having $k$-homology in an object imply $a$ and $b$ homologies in (complementary) sub-objects, where $a+b=k-2$? The rest of this paper is motivated by considering this question in the various settings offered by \cref{p:equivalences}.

%%%%%%%%%%%%%%%%%%%%%%%%
\section{Complementary homologies of lattices} \label{sec:3}
%%%%%%%%%%%%%%%%%%%%%%%%

In this section we provide the proof of \cref{thm:complement}. Along the way
we will derive in \cref{lem:homotopic} an unexpected homotopy equivalence related to complementation
in lattices. The next lemma is a first step towards \cref{lem:homotopic}.

\begin{lemma}\label{lem:join}
Let $x\in \overline{L}$ and $\rhk_{k}(\,\overline{L}; \field\,) \neq 0$ for some $k$. Then, there exists  
$y\in \MC(x)$ such that
\[
\rhk_{k-1}\big(\,(\hat{0},y)_L* \big((y,\hat{1})_L\setminus \MC(x)\big);\field\,\big) \neq 0
\]
\end{lemma}

\begin{proof}
    By \cite{BW83complements}, the poset $L_0 = \overline{L}\setminus \MC(x)$ is contractible. Write $\MC(x) = \{y_1,y_2,\ldots,y_r\}$ such that if $i<j$, then either $y_i<y_j$ or $y_i$ is incomparable to $y_j$. For $1\leq j \leq r,$ define $L_j = \big(\overline{L}\setminus \MC(x)\big) \cup \{y_1,\ldots,y_j\}.$ Note that $L_r = \overline{L}$. Since $\rhk_{k}(\overline{L}; \field) \neq 0$, there exists a $j$ with $1\leq j \leq r$ such that 
    \[
    \rhk_{k}(L_i; \field) = 0 \qforall 0\leq i < j, \qand  \rhk_{k}(L_j; \field) \neq  0.
    \]
    
    Consider the vertex $y_j$ in the poset $L_j$ and note that $\del_{L_j}(y_j) = L_{j-1}$ and $\lk_{L_j}(y_j) = (\hat{0},y_j)_{L_j}*(y_j,\hat{1})_{L_j}$. By the ordering of $\MC(x),$ we have $(\hat{0},y_j)_{L_j} = (\hat{0},y_j)_{L}$ and $(y_j,\hat{1})_{L_j} = (y_j,\hat{1})_{L}\setminus \MC(x)$. The long exact sequence of link and deletion gives:
    \[
    \cdots \rightarrow \  \rhk_{k}(\lk_{L_j}(y_j); \field) \ \rightarrow \ \rhk_{k}(\del_{L_j}(y_j); \field) \ \rightarrow \ \rhk_{k}(L_j; \field)
   \ \rightarrow \ \rhk_{k-1}(\lk_{L_j}(y_j); \field) \ \rightarrow \cdots.
    \]
    By the choice of $j$ we have $$\rhk_{k}(\del_{L_j}(y_j); \field) =0, \, \rhk_{k}(L_j; \field) \neq  0 \text{ and } \lk_{L_j}(y_j) = (\hat{0},y_j)_{L} * (y_j,\hat{1})_{L}\setminus \MC(x).$$
    The result now follows from the long exact sequence.
\end{proof}

Next we want to understand the topology of the poset $(y,\hat{1})_L\setminus \MC(x)$. 
We apply Quillen's Fiber Lemma (\cref{lem:quillen}) to compare two naturally arising subposets associated to complementary elements.

\begin{proposition}\label{lem:homotopic}
    Let $x\in \overline{L}$ and $y\in \MC(x)$. Then, 
    \[
    (\hat{0},x)_L\setminus \MC(y) \simeq (y,\hat{1})_L\setminus \MC(x).
    \]
\end{proposition}

\begin{proof}
Let 
    $$P=(\hat{0},x )_L\setminus \MC(y) 
    \qand 
    Q= (y,\hat{1})_L\setminus \MC(x).$$

First, consider the case $P = \emptyset$. We show that $Q = \emptyset$.
    On the contrary, suppose that $Q \neq \emptyset$, and let $z \in Q$. 
    Then,
    $ \hat{1} = y \vee x \leq z \vee x \leq  \hat{1}$. 
    Since $z\notin \MC(x)$, it follows that $z \wedge x \neq \hat{0}$. Hence, from $z \vee x = \hat{1}$, we obtain
    $\hat{0} < 
    z \wedge x < x$. 
    Moreover, from $P = \emptyset$ it follows that $z \wedge x \in \MC(y)$, and therefore, $y \vee (z \wedge x) = \hat{1}$. 
    On the other hand, since $y< z< \hat{1}$ and $z \wedge x< z$, we obtain $y \vee (z \wedge x) \leq z < \hat{1}$ which is a contradiction. Hence, $Q = \emptyset$.
    
    It remains to consider the case $P \neq \emptyset$.
    If $x'\in P$ then $x' < x$. Thus $x \wedge y = \hat{0}$ implies
    $x' \wedge y =\hat{0}$. Since $x' \not\in \MC(y)$ we must have
    $x' \vee y < \hat{1}$.
    Moreover, for all $x'\in P$, we have $x'\vee y \notin \MC(x)$ because $(x'\vee y)\wedge x  \geq x'$.
    Thus, $x' \vee y \in Q$ for all $x'\in P$.
    Hence we can define a map 
    $$\phi:P \to Q \qwhere \phi(x') = x'\vee y$$ and it follows from the definition of join that $\phi$ is order preserving. 
   
  Note that $P$ may not be a lattice because the join of two elements in $P$ does not necessarily lie in $P$. We claim that, for $z \in Q$,  
   \begin{equation}\label{eq:vee}
   x_1,x_2\in \phi^{-1}(Q_{\leq z}) \Longrightarrow 
   x_1\vee x_2\in \phi^{-1}(Q_{\leq z}).
   \end{equation}
   Since $x_1, x_2\leq x$ and $L$ is a lattice, $x_1 \vee x_2 \leq x$.
   It remains to show that $x_1 \vee x_2 \notin \MC(y)$, which would imply in particular that $x_1 \vee x_2 < x$.
  
    As  $\phi(x_1)\leq z$ and $\phi(x_2)\leq z$,  we have 
    $$
    x_1\vee x_2 \vee y = 
    (x_1\vee y) \vee (x_2 \vee y) \leq z < \hat{1}
    $$  
    and so $x_1\vee x_2 \vee y \neq \hat{1}$, proving the claim in~\eqref{eq:vee}.
    Since $z\notin \MC(x)$ and $z > y \in \MC(x)$ it follows that $z \wedge x > \hat{0}$.
    By $z \geq z \wedge x,y$ we infer that $z \wedge x \not \in \MC(y)$.
    This implies both
    $z \wedge x \in P$ and  $\phi(z \wedge x) = (z \wedge x) \vee y \leq z$. 
    From that we get that $\phi^{-1}(Q_{\leq z})$ is non-empty. 
    
     We now show that $\phi^{-1}(Q_{\leq z})$ has a unique maximal element. On the contrary, if $x_1$ and $x_2$ are two maximal elements in $\phi^{-1}(Q_{\leq z})$, then by \eqref{eq:vee}  $x_1\vee x_2 \in \phi^{-1}(Q_{\leq z})$ contradicting
the maximality of $x_1$ and $x_2$. Hence, $\phi^{-1}(Q_{\le z})$ is a cone with apex the unique maximal element and therefore contractible.
By \cref{lem:quillen}, $\phi$ induces a homotopy equivalence, proving the claim.
\end{proof}

When the set of complements of an $x\in \overline{L}$ forms an antichain, Bj\"orner and Walker~\cite{BW83complements} proved a homotopy formula for $\overline{L}$ in terms of intervals 
determined by elements of $\MC(x)$. Using \Cref{lem:homotopic}, we obtain the following reformulation of their result.

\begin{corollary}\label{cor:antichain}
  Let $L$ be a lattice and $x\in \overline{L}$ be such that $\MC(x)$ is an antichain. Then,
  \[
\overline{L} \quad \simeq \bigvee_{y \in \MC(x)} \susp\big(\, (\hat{0},y)* (\hat{0}, x) \setminus \MC(y)\,\big).
\]
\end{corollary}
 
\begin{proof}
By \cite{BW83complements}, 
\[
\overline{L} \, \simeq \bigvee_{y \in \MC(x)} \susp\big(\, (\hat{0},y)* (y, \hat{1})\,\big).
\]
Since $\MC(x)$ is an antichain, we get $(y, \hat{1})= (y, \hat{1}) \setminus \MC(x)$.
Therefore, by \Cref{lem:homotopic},
\[
\overline{L} \, \simeq \bigvee_{y \in \MC(x)} \susp\big(\, (\hat{0},y)* (\hat{0}, x) \setminus \MC(y)\,\big).
\]
\end{proof}

The following lemma shows how non-trivial homology in the subposet $(\hat{0},x)_L\setminus \MC(y)$ 
disseminates in the poset.

\begin{lemma}\label{lem:openhomology}
     Let $x\in \overline{L}$ and $y\in \MC(x)$. If $\rhk_{a}\big(\,(\hat{0},x )_L\setminus \MC(y);\field\,\big) \neq 0$ for some $a$, then there exists  $z \in (\hat{0},x]_L\cap \MC(y)$ such that $\rhk_{a}\big(\,(\hat{0},z )_L;\field\,\big) \neq 0$. 
\end{lemma}  

\begin{proof}
    If $\rhk_{a}\big(\,(\hat{0},x )_L;\field\,\big) \neq 0$, then we can choose $z=x$. It remains to consider the case $\rhk_{a}\big(\,(\hat{0},x )_L;\field\,\big) = 0$.
    Let $L_0= (\hat{0},x )_L\setminus \MC(y)$.
    Write $(\hat{0},x )_L\cap \MC(y) = \{y_1,y_2,\ldots,y_r\}$ such that if $i<j$, then either $y_i<y_j$ or $y_i$ is incomparable to $y_j$. For $1 \leq j \leq r$, define 
    $$L_j = \big(\,(\hat{0},x )_L\setminus \MC(y)\,\big) \cup \{y_1,\ldots,y_j\}.$$ Since $\rhk_{a}\big(\,(\hat{0},x )_L\setminus \MC(y);\field\,\big) \neq 0$ and $\rhk_{a}\big(\,(\hat{0},x )_L;\field\,\big) = 0$, there exists a $j$ with $1\leq j \leq r$ such that 
    \[
    \rhk_{a}(L_{j-1}; \field) \neq 0, \qand  \rhk_{a}(L_{j}; \field) =  0.
    \]

    Now consider the vertex $y_j$ in the poset $L_j$. Note that $\del_{L_j}(y_j) = L_{j-1}$ and $\lk_{L_j}(y_j) = (\hat{0},y_j)_{L_j}$. By the chosen ordering of $(\hat{0},x )_L\cap \MC(y),$ we have $(\hat{0},y_j)_{L_j} = (\hat{0},y_j)_{L}$. The long exact sequence of link and deletion gives:
    \[
    \cdots  \ \rightarrow \ \rhk_{a+1}(L_j; \field) \rightarrow \  \rhk_{a}(\lk_{L_j}(y_j); \field) \ \rightarrow \ \rhk_{a}(\del_{L_j}(y_j); \field) \ \rightarrow \ \rhk_{a}(L_j; \field) \ \rightarrow \cdots.
    \]
    Since  $\rhk_{a}(\del_{L_j}(y_j); \field) = \rhk_{a}(L_{j-1}; \field) \neq 0$ and $\rhk_{a}(L_{j}; \field) =  0$, it follows that  
    $$\rhk_{a}\big(\,(\hat{0},y_j)_L;\field\,\big) = 
    \rhk_{a}\big(\lk_{L_j}(y_j);\field\,\big) \neq 0\,.
    $$
    Thus, the desired element $z$ is given by $z = y_j$.
\end{proof}

We now are in position to prove our main result \cref{thm:complement}.

\begin{proof}[Proof of \cref{thm:complement}]
Let $x \in \overline{L}$. By {\Cref{lem:join}}, there exists $y\in \MC(x)$ such that 
\[
\rhk_{k-1}\Big(\,(\hat{0},y)_L* \big(\,(y,\hat{1})_L\setminus \MC(x)\,\big);\field\,\Big) \neq 0.
\]
By the formula for the homology of a join~(see e.g., \cite[page\ 1847]{Bjornertopologicalmethods}),
we have  
\[
\rhk_{k-1}\Big(\,(\hat{0},y)_L* \big(\,(y,\hat{1})_L\setminus \MC(x)\,\big);\field\,\Big) = \mathop{\bigoplus}_{i+j=k-2} \rhk_{i}\big(\,(\hat{0},y)_L;\field\,\big) \otimes \rhk_{j}\big(\,(y,\hat{1})_L\setminus \MC(x);\field\,\big).
\]

Therefore, there exist $a, b$ such that $a+ b = k-2$ and
\[
\rhk_{a}\big(\,(y,\hat{1})_L\setminus \MC(x); \field\,\big) \neq 0  \qand \rhk_{b}\big(\,(\hat{0},y )_L; \field\, \big) \neq 0.
\]
By \Cref{lem:homotopic}, $(y,\hat{1})_L\setminus \MC(x) \simeq (\hat{0},x)_L\setminus \MC(y)$; thus,
\[
\rhk_{i}\big(\,(y,\hat{1})_L\setminus \MC(x); \field \,\big) \iso \rhk_{i}\big(\,(\hat{0},x)_L\setminus \MC(y);\field\, \big) \qforall i.
\]
Therefore, by \Cref{lem:openhomology}, there exists $z \in (\hat{0},x]_L\cap \MC(y)$ such that $\rhk_{a}\big(\,(\hat{0},z )_L;\field\,\big) \neq 0$. This completes the proof of the main statement. 

Now suppose that $\rhk_{b'}\big(\,(\hat{0},y)_L;\field\,\big) \neq 0$ for some $b'$.  
By the homology formula for joins, we obtain
\[
\rhk_{b'+a+1}\Big(\, (\hat{0},y) * \big(\,(\hat{0}, x) \setminus \MC(y)\,\big),
\field\,\Big) \neq 0.
\]
Therefore, 
\[
\rhk_{b'+a+2}\Big(\,\susp\!\big(\, (\hat{0},y)_L * \big((\hat{0}, x)_L \setminus \MC(y)\big)\,\big),
\field\,\Big) \neq 0.
\]
Since $\MC(x)$ is an antichain, by \Cref{cor:antichain} we obtain
\[
\rhk_{b'+a+2}\big(\,\overline{L}; \field\,\big) \neq 0.
\]
\end{proof}

The  minimal elements of $\overline{L}$ do have complements with complementary homologies, as stated in \cref{que:main}.

\begin{corollary}\label{cor:1complemented}
    If $\rhk_{k}(\overline{L}; \field) \neq 0$ for some $k\geq 0$ and $x$ is a minimal element of $\overline{L}$, then there is $y \in \MC(x)$ such that $\rhk_{k-1}\big(\,(\hat{0},y )_L;\field \big) \neq 0$. 
\end{corollary}
\begin{proof}
    If $x$ is a minimal element of $\overline{L}$, then $(\hat{0}, x)_L = \emptyset$ which implies that $(\hat{0}, x]_L \cap \MC(y) = \{ x\}$. Therefore, in
    \Cref{thm:complement}
    we have $z = x$. Since $\rhk_{i}\big(\,(\hat{0},x )_L;\field\, \big) \neq 0$ if and only if $i=-1$ it follows that $a = -1$, which then implies $b = k+1$.
    Now the corollary follows from \Cref{thm:complement}.
\end{proof}

In the setting of \Cref{lem:openhomology}, it is natural to ask whether one can obtain the additional conclusion $\rhk_{a}\big((\hat{0},z)_L\setminus \MC(y);\field\big)\neq 0$.
The following example shows that this is not always the case.

\begin{example}\label{ex:complement}
Consider the lattice $L$ whose elements are all unions of the sets
\begin{align*}
&S_1 = \{1,5,9\}, &&S_2 = \{2,5,8\}, &&S_3 = \{2,7,8\}, &&S_4 = \{2,4,9\},\\
&S_5 = \{4,5,6\}, &&S_6 = \{3,7,8\}, &&S_7 = \{6,7,8\}, &&S_8 = \{4,7,8\}
\end{align*}

with the empty set a the unique minimal element $\hat{0}$, as shown in \cref{fig:complement}
\begin{figure}
\begin{tikzpicture}[
  x=0.7cm,
  y=0.7cm,
  every node/.style={
    draw,
    rounded corners=1.5pt,
    fill=white,
    inner xsep=3pt,
    inner ysep=2pt
  }
]

  \node (v0) at (0,0.0) {$\scriptstyle \emptyset$};

  \node (v1) at (-5,2.5) {$\scriptstyle 1,5,9$};
  \node (v2) at (-2.5,2.5) {$\scriptstyle 2,4,9$};
  \node (v3) at (2.5,2.5) {$\scriptstyle 2,5,8$};
  \node (v4) at (5,2.5) {$\scriptstyle 4,5,6$};

  \node[text=blue,draw=blue] (v5) at (-8,6) {$\scriptstyle 1,2,4,5,9$};
  \node (v6) at (-5,6) {$\scriptstyle 1,2,5,8,9$};
  \node (v7) at (-2.5,6) {$\scriptstyle 1,4,5,6,9$};
  \node (v10) at (1,6) {$\scriptstyle 2,4,5,8,9$};
  \node (v9) at (4,6) {$\scriptstyle 2,4,5,6,9$};
  \node (v8) at (7,6) {$\scriptstyle 2,4,5,6,8$};

  \node[text=blue,draw=blue] (v11) at (-0.5,9) {$\scriptstyle 1,2,4,5,6,9$};
  \node[text=blue,draw=blue] (v12) at (-7,9) {$\scriptstyle 1,2,4,5,8,9$};
  \node (v13) at (6,9) {$\scriptstyle 2,4,5,6,8,9$};

  \node[text=blue,draw=blue] (v14) at (0,11.8) {$\scriptstyle 1,2,4,5,6,8,9$};

    \draw (v0) -- (v1);
    \draw (v0) -- (v2);
    \draw (v0) -- (v3);
    \draw (v0) -- (v4);
    \draw (v1) -- (v5);
    \draw (v1) -- (v6);
    \draw (v1) -- (v7);
    \draw (v2) -- (v5);
    \draw (v2) -- (v9);
    \draw (v2) -- (v10);
    \draw (v3) -- (v6);
    \draw (v3) -- (v8);
    \draw (v3) -- (v10);
    \draw (v4) -- (v7);
    \draw (v4) -- (v8);
    \draw (v4) -- (v9);
    \draw[draw=blue,thick] (v5) -- (v11);
    \draw[draw=blue,thick] (v5) -- (v12);
    \draw (v6) -- (v12);
    \draw (v7) -- (v11);
    \draw (v8) -- (v13);
    \draw (v9) -- (v11);
    \draw (v9) -- (v13);
    \draw (v10) -- (v12);
    \draw (v10) -- (v13);
    \draw[draw=blue,thick] (v11) -- (v14);
    \draw[draw=blue,thick] (v12) -- (v14);
    \draw (v13) -- (v14);
\end{tikzpicture}
\caption{Interval $[\hat{0},X]_L$ in \cref{ex:complement}} \label{fig:complement}
\end{figure}

Let 
$$X = \{1,2,4,5,6,8,9\} = S_1 \cup S_2 \cup S_4 \cup S_5 
\qand 
Y = \{3,6,7,8\}\,= S_6 \cup S_7.$$ 
Since $|X \cap Y | = 2$ none of the $S_i$ is contained in both $X$ and $Y$. It follows
that $X \wedge Y = \hat{0}$, and since $X \cup Y$ is the union of all $S_i$
it follows that $X \vee Y = \hat{1}$ and hence $Y \in \MC(X)$. 

Since $1$ is only contained in $S_1$, any complement $Z$ of $Y$ must contain 
$S_1$. The set $S_1$ itself is not a complement of $Y$ since $2 \not\in S_1 \cup Y$. 
So any complement of $Y$ in the interval $(\hat{0},X]_L$  must be in the
interval $(S_1,X]_L$. 
From \cref{fig:complement} we get that the interval $(S_1,X]_L$ consists of
$\{1,2,4,5,9\}, \{1,2,5,8,9\}$, $\{1,4,5,6,9\}$, $\{1,2,4,5,8,9\}$,
$\{1,2,4,5,6,9\}$ and $X$. It can be  easily checked that 
\[
\MC(Y)\cap (\hat{0},X ]_L=
\big \{ 
Z_1=\{1,2,4,5,8,9\}, 
\  
Z_2=\{1,2,4,5,6,9\}, 
\ 
Z_3=\{ 1,2,4,5,9\}, 
\ X
\big \}\,;
\]
these elements are labeled blue in \cref{fig:complement}.

One can see from \cref{fig:complement} that the  $1$-cycle
\begin{align*}
     &\{S_4, \{2,4,5,8,9\}\} - \{S_2, \{2,4,5,8,9\}\} + \{S_2, \{1, 2,5,8,9\}\} - \{S_1,\{1, 2,5,8,9\} \} +\\
    & 
    \{S_1,\{1, 4,5,6,9\}\} - \{S_5,\{1, 4,5,6,9\}\}  + \{S_5,\{2, 4,5,6,9\}\} -  \{S_4, \{2,4,5,6,9\}\} 
\end{align*}
is a non-trivial in $(\hat{0}, X)_L\setminus \MC(Y)$. Therefore,
\[
\rhk_{1}\big(\,(\hat{0}, X)_L\setminus \MC(Y);\field\,\big) \neq 0\,.
\]
We will now show that for every $Z\in \MC(Y)\cap (\hat{0},X ]_L$, 
$$
    \rhk_{1}\big(\,(\hat{0}, Z)_L;\field\,\big) = 0 
    \qor 
    \rhk_{1}\big(\,(\hat{0},Z )_L\setminus \MC(Y);\field\,\big) = 0\,.
    $$

 We first consider the case $Z=X$: since the element $S_4=\{2,4,9\}$ in \cref{fig:complement} does not have a complement in the lattice $[\hat{0},X]_L$, the interval $(\hat{0},X)_L$ is acyclic by~\cite[Corollary~6.3]{Bac77}. 
For $Z\in \{Z_1,Z_2\}$ one can see from \cref{fig:complement} that the interval $[\hat{0}, Z]$ is a  Boolean lattices of $3$ elements; thus the  open interval $(\hat{0}, Z)$ is homeomorphic to $S^1$. 
However one can see from \cref{fig:complement} that $(\hat{0}, Z)\setminus \MC(Y)$ is acyclic. Finally, if $Z=Z_3$, then $(\hat{0}, Z)$ is two points, and can only have $0$-homology. 

\end{example}

We close this section with a counterexample to the following strengthening of
\cref{que:main}.

\begin{question} \label{que:strongmain}
Let $L$ be a lattice and suppose that
\begin{itemize}
\item $\rhk_k\big(\,\overline{L};\KK\,\big) \neq 0$ for some $k \geq 0$,
\item $a,b \geq 0$ and $a+b = k-2$.
\end{itemize}
Then for any $x \in L$ with
$$
\rhk_{a}\big(\,(\hat{0},x)_L;\KK\,\big) \neq 0,$$ 
is there a $y \in \MC(x)$ such that 
$\rhk_{b}\big(\,(\hat{0},y)_L;\KK\,\big) \neq 0$?
\end{question}

The following example shows that the answer is no.

\begin{example}\label{example:counter-example}
Consider the lattice $L$ whose elements are all unions of the sets
     \begin{align*}
     S_1 = \{1,2,5\}, &&S_2 = \{2,5,7\},  &&S_3 = \{1,4,10\}, &&S_4 = \{1,5,9\},\\
     S_5 = \{2,4,10\}, &&S_6 = \{3,7,8\}, &&S_7 = \{5,6,10\} , &&     S_8 = \{4,5,8\}\,,
     \end{align*}
     with the empty set as the unique minimal element $\hat {0}$. 
     The set $X = \{1,2,5,7\} = S_1 \cup S_2$ only contains $S_i$ for $i =1,2$. It follows
     that $\rhk_i\big(\,(\hat{0},X)_L;\KK\,\big) = 0$ for $i \neq 0$ and 
     $= \KK$ for $i = 0$. 
     
\begin{figure}[htbp!]
  \centering
    \begin{tikzpicture}[
  x=0.5cm,y=0.9cm,
  vertex/.style={draw,rounded corners=1pt,fill=white,inner sep=1.6pt,
                 font=\scriptsize,align=center},
  edge/.style={line width=0.3pt}
]
\coordinate (p0) at (0,0);
\coordinate (p1) at (-8,3);
\coordinate (p2) at (-4,3);
\coordinate (p3) at (3.5,3);
\coordinate (p4) at (9,3);
\coordinate (p5) at (0,3);
\coordinate (p6) at (-6,6);
\coordinate (p7) at (-0.6,6);
\coordinate (p8) at (3.5,6);
\coordinate (p9) at (-9,6);
\coordinate (p10) at (-12,6);
\coordinate (p11) at (10,6);
\coordinate (p12) at (7.5,6);
\coordinate (p13) at (6,9.5);
\coordinate (p14) at (1,9);
\coordinate (p15) at (-2.8,9);
\coordinate (p16) at (-12,9);
\coordinate (p17) at (-6.4,9);
\coordinate (p18) at (10,9.2);
\coordinate (p19) at (-2,12);
\coordinate (p20) at (3,10.5);
\coordinate (p21) at (-12,13);
\coordinate (p22) at (7.7,11.1);
\coordinate (p23) at (4,15);
\coordinate (p24) at (-4,15);
\coordinate (p25) at (0,13.7);
\coordinate (p26) at (0,18);
\coordinate (au1) at (-7,12);
\coordinate (au2) at (5.5,6);
\coordinate (au3) at (6,6);
\draw[edge,draw=brown,very thick] (p0) -- (p1);
\draw[edge] (p0) -- (p2);
\draw[edge,draw=brown,very thick] (p0) -- (p3);
\draw[edge] (p0) -- (p4);
\draw[edge,draw=brown,very thick] (p0) -- (p5);
\draw[edge] (p1) -- (p7);
\draw[edge] (p1) -- (p9);
\draw[edge,draw=blue,thick] (p1) -- (p10);
\draw[edge,draw=blue,thick] (p1) -- (p14);
\draw[edge] (p2) -- (p6);
\draw[edge] (p2) -- (p8);
\draw[edge] (p2) -- (p9);
\draw[edge] (p2) -- (p13);
\draw[edge] (p3) -- (p11);
\draw[edge,draw=blue,thick] (p3) -- (p14);
\draw[edge] (p3) -- (au2);
\draw[edge] (au2) -- (p13);
\draw[edge,draw=blue,thick] (p3) -- (au3);
\draw[edge,draw=blue,thick] (au3) -- (p18);
\draw[edge] (p4) -- (p7);
\draw[edge] (p4) -- (p8);
\draw[edge] (p4) -- (p11);
\draw[edge] (p4) -- (p12);
\draw[edge] (p5) -- (p6);
\draw[edge,draw=blue,thick] (p5) -- (p10);
\draw[edge] (p5) -- (p12);
\draw[edge,draw=blue,thick] (p5) -- (p18);
\draw[edge] (p6) -- (p15);
\draw[edge] (p6) -- (p16);
\draw[edge] (p7) -- (p17);
\draw[edge] (p7) -- (p19);
\draw[edge] (p8) -- (p15);
\draw[edge] (p8) -- (p17);
\draw[edge] (p8) -- (p20);
\draw[edge] (p9) -- (p16);
\draw[edge] (p9) -- (p17);
\draw[edge] (p10) -- (p16);
\draw[edge,draw=brown,very thick] (p10) -- (au1);
\draw[edge,draw=brown,very thick] (au1) -- (p25);
\draw[edge] (p11) -- (p19);
\draw[edge] (p11) -- (p20);
\draw[edge] (p11) -- (p22);
\draw[edge] (p12) -- (p15);
\draw[edge] (p12) -- (p22);
\draw[edge] (p13) -- (p20);
\draw[edge] (p14) -- (p19);
\draw[edge,draw=brown,very thick] (p14) -- (p25);
\draw[edge] (p15) -- (p21);
\draw[edge] (p15) -- (p23);
\draw[edge] (p16) -- (p21);
\draw[edge] (p17) -- (p21);
\draw[edge] (p17) -- (p24);
\draw[edge] (p18) -- (p22);
\draw[edge,draw=brown,very thick] (p18) -- (p25);
\draw[edge] (p19) -- (p24);
\draw[edge] (p20) -- (p23);
\draw[edge] (p20) -- (p24);
\draw[edge] (p21) -- (p26);
\draw[edge] (p22) -- (p23);
\draw[edge] (p23) -- (p26);
\draw[edge] (p24) -- (p26);
\draw[edge] (p25) -- (p26);
\node[vertex] at (p0) {$\emptyset$};
\node[vertex,text=blue] at (p1) {$\scriptstyle{1,5,9}$};
\node[vertex] at (p2) {$\scriptstyle{1,4,10}$};
\node[vertex,text=blue] at (p3) {$\scriptstyle{3,7,8}$};
\node[vertex,text=red] at (p4) {$\scriptstyle{4,5,8}$};
\node[vertex,text=blue] at (p5) {$\scriptstyle{5,6,10}$};
\node[vertex] at (p6) {$\scriptstyle{1,4,5,6,10}$};
\node[vertex] at (p7) {$\scriptstyle{1,4,5,8,9}$};
\node[vertex] at (p8) {$\scriptstyle{1,4,5,8,10}$};
\node[vertex] at (p9) {$\scriptstyle{1,4,5,9,10}$};
\node[vertex,text=blue] at (p10) {$\scriptstyle{1,5,6,9,10}$};
\node[vertex] at (p11) {$\scriptstyle{3,4,5,7,8}$};
\node[vertex] at (p12) {$\scriptstyle{4,5,6,8,10}$};
\node[vertex] at (p13) {$\scriptstyle{1,3,4,7,8,10}$};
\node[vertex,text=blue] at (p14) {$\scriptstyle{1,3,5,7,8,9}$};
\node[vertex] at (p15) {$\scriptstyle{1,4,5,6,8,10}$};
\node[vertex] at (p16) {$\scriptstyle{1,4,5,6,9,10}$};
\node[vertex] at (p17) {$\scriptstyle{1,4,5,8,9,10}$};
\node[vertex,text=blue] at (p18) {$\scriptstyle{3,5,6,7,8,10}$};
\node[vertex] at (p19) {$\scriptstyle{1,3,4,5,7,8,9}$};
\node[vertex] at (p20) {$\scriptstyle{1,3,4,5,7,8,10}$};
\node[vertex] at (p21) {$\scriptstyle{1,4,5,6,8,9,10}$};
\node[vertex] at (p22) {$\scriptstyle{3,4,5,6,7,8,10}$};
\node[vertex] at (p23) {$\scriptstyle{1,3,4,5,6,7,8,10}$};
\node[vertex] at (p24) {$\scriptstyle{1,3,4,5,7,8,9,10}$};
\node[vertex,text=brown] at (p25) {$\scriptstyle{1,3,5,6,7,8,9,10}$};
\node[vertex] at (p26) {$\scriptstyle{1,3,4,5,6,7,8,9,10}$};
\end{tikzpicture}
\caption{Interval $(\hat{0},Y)_L$ in \cref{example:counter-example}} \label{fig:interval}
\end{figure}

Let $Y$ be a complement of $S_1 \cup S_2$ in $L$. Then $Y$ must contain $3$ and $9$. 
     But $3$ is only contained in $S_6$ and $9$ in only contained in $S_4$. Thus
     $S_4 \cup S_6 = \{1,3,5,7,8,9\} \subseteq Y$. Clearly, $S_4 \cup S_6$ is not yet a
     complement of $S_1 \cup S_2$. By
     $S_2 \subseteq S_4 \cup S_5 \cup S_6, S_4 \cup S_6 \cup S_7$ it follows that 
     $S_5,S_7 \not\subseteq Y$ which leaves only $S_3$ and $S_8$ to complete $Y$ to a complement of $X$. 
     We have 
     \[S_3 \cup S_4 \cup S_6 =  \{1,3,4,5,7,8,9,10\} \qand 
     S_4 \cup S_6 \cup S_8 =\{1,3,4,5,7,8,9\}\,.
     \]
     Since $S_3 \cup S_4 \cup S_6$ lacks $10$ we have that 
     $Y = S_4 \cup S_6 \cup S_8 =\{1,3,4,5,7,8,9\}$
     is the unique complement of $X$ in $L$. 
      Since $Y$ is maximal in $\overline{L}$ it follows from the homotopy complementation formula 
      \cite[Theorem 4.2]{BW83complements} that $\overline{L}$ is homotopy equivalent to a suspension of $(\hat{0},Y)_L$. 

     In the interval $(\hat{0},Y)_L$, shown in \cref{fig:interval}, the only complement to
     $X'=\{4,5,8\}$ is the maximal element $Y' = \{1,3,5,6,7,8,9,10\}$. Thus again by
     the homotopy complementation formula \cite[Theorem 4.2]{BW83complements} it follows that 
     $(\hat{0},Y)_L$ is homotopy equivalent to the suspension of 
     $(\hat{0},Y')_L$. In \cref{fig:interval} the elements of $(\hat{0},Y')_L$ colored blue. It follows that the dimension of $(\hat{0},Y')_L$ is $1$ and 
     its homology is
     concentrated in the homological dimension $1$, where it is one dimensional. Hence the
     $\rhk_i\big(\,(\hat{0},Y)_L;\KK\,\big) = 0$ for $i \neq 2$ and $=\KK$ for $i=2$. 

     As a consequence we get that $\rhk_k\big(\,\overline{L};\KK\,\big) = 0$ for $k \neq 3$ and
     $=\KK$ for $k =3$. Now let $a=0$ and $b=1$ then the above shows
     \begin{itemize}
         \item $a+b=k-2$,
         \item $\rhk_k\big(\,\overline{L};\KK\,\big) \neq 0$,
         \item $\rhk_{a}\big(\,(\hat{0},X)_L;\KK\,\big) =\KK$,
         \item $\rhk_{b}\big(\,(\hat{0},Y)_L;\KK\,\big) = 0$ for the unique $Y \in \MC(X)$.
         \end{itemize}
\end{example}

%%%%%%%%%%%%%%%%%%%%%%%%%%%%%%%%%%%%%%
\section{Complementary faces of simplicial complexes}\label{s:faces}
%%%%%%%%%%%%%%%%%%%%%%%%%%%%%%%%%%%%%%

This section is dedicated to the proof of \cref{thm:simplicialcomplex}, which provides affirmative evidence towards \cref{que:main}. 
We begin by defining the notion of complementary faces in a simplicial complex. We then establish a result concerning the existence of a cycle in a non-acyclic simplicial complex whose support contains complementary faces.

Given a simplicial complex $\Delta$, let $L=\ML(\Delta)$ be the face lattice of $\Delta$. We call two faces $\sigma, \tau \in \Delta$  {\it complements} if they are complements as elements of the lattice $L$. Equivalently, $\sigma$ and $\tau$ are complements if 
$$
\sigma \cup \tau \notin \Delta 
\qand 
\sigma \cap \tau = \emptyset.
$$
An immediate consequence of \cref{lem:intlattice} is the following reformulation of Bj\"orner and Walker's complementation theorem.

\begin{theorem}\label{thm:noncomplemented-contractible}
     If a simplicial complex $\Delta$ is not complemented (equivalently, if $\ML(\Delta)$ is
     not a complemented lattice), then $\Delta$ is contractible.
\end{theorem}

We write a simplicial $k$-cycle $\gamma$ for the simplicial complex as $\gamma=c_1\rho_1+\cdots+c_s\rho_s$ where all $c_i \in \field \setminus \{0\}$ and all the $\rho_i$ are distinct $k$-dimensional faces of $\Delta$. 
By the \textit{support} $\supp(\gamma)$ of $\gamma$ we mean the simplicial complex $\langle \rho_1,\ldots,\rho_s \rangle$.

\begin{lemma}\label{p:almost-complements} 
Let $\Delta$ be a simplicial complex such that $\rhk_{k}(\Delta;\field) \neq 0$ for some $k\geq 0$. Then, there exists a $k$-cycle $\gamma$ that is not a boundary such that for every $a,b \geq 0$ with $a+b=k$, there exist complementary faces $\sigma$ and $\tau$  in $\supp{\gamma}$ of dimensions $a$ and  $b$, respectively.
\end{lemma}

\begin{proof} 
Fix a linear order on the vertex set of $\Delta$. Extend this linear order to a lexicographic order on 
the $k$-dimensional faces. For each $k$-cycle $\alpha$ of $\Delta$ let $\min(\alpha)$ denote the smallest $k$-face of $\alpha$ under this order. We now fix $\rho$ to be the largest   $\min(\alpha)$ among all $k$-cycles $\alpha$ of $\Delta$ which are not boundaries, and let $\gamma$ be the corresponding $k$-cycle of $\Delta$, in other words
\begin{equation}\label{eq:gamma}
\gamma=c_1 \rho_1 + \cdots + c_s \rho_s
\qwhere
\rho=\rho_1.
\end{equation}
We claim that there exist vertices $x\in \rho$ and $y_x \in \Delta$ such that 
\begin{equation}\label{eq:assumption}
\big ( \rho \setminus \{x\} \big ) \cup \{y_x\} \in \supp(\gamma)
\qand 
\rho \cup \{y_x\} \notin \Delta. 
\end{equation}

Assume, to the contrary, that for every vertex $x\in \rho$ and $y_x \in \Delta$ such that $\big ( \rho \setminus \{x\} \big ) \cup \{y_x\} \in \supp(\gamma)$, we also have $\rho \cup \{y_x\} \in \Delta$.

Since for all $x \in \rho$ there is a vertex $y_x$ such that  $\rho \setminus \{x\} \cup \{y_x\} \in \supp(\gamma)$, the minimality of $\rho$ and the lexicographic order on the vertices of $\supp(\gamma)$ imply that $y_x>x$. 

By our assumption in \eqref{eq:assumption} $\rho \cup \{y_x\} \in \Delta$ for all $x \in \rho$, and in particular, $\rho$ is contained in a $(k+1)$-face of $\Delta$. 
Let $\omega = \rho \cup \{z\}$ be the maximal
$(k+1)$-face in $\Delta$ containing $\rho$. Since $y_x > x$ for all $x \in \rho$ we have that  $z$ is larger than the maximal element of $\rho$. 
Choose $u$ to be the minimal element of 
$\rho$. Then observe that $u$ and $z$ are, respectively, the smallest and the largest vertices in $\omega$.  

Now $\rho \in \supp(\partial(\omega))$, and using \eqref{eq:gamma} 
$$
\gamma'=\gamma - (-1)^k c_1\partial(\omega)
$$
is a homology $k$-cycle homologous to $\gamma$ and not containing $\rho$ in its support. By the choice of $\rho$ the cycle $\gamma'$ must contain a $k$-face $\rho'$ in its support which is
lexicographically smaller than $\rho$. 
Then, since $\rho$ is the smallest $k$-face in $\supp(\gamma)$, we must have $\rho' \in \partial(\omega)=\partial(\rho \cup \{z\})$ which implies that 
$$
\rho'=\big ( \rho \setminus \{x'\}\big ) \cup \{z\} 
\qforsome
x' \in \rho.
$$ 
On the other hand, since $z> x$ for all $x \in \rho$, we have $\rho' >\rho$. 
This contradicts our choice of $\rho'$ being lexicographically smaller than $\rho$. 

Therefore, there must be  vertices $x\in \rho$ and $y_x \in \Delta$ such that \eqref{eq:assumption} holds, or  more specifically 
$$
\big ( \rho \setminus \{x\} \big ) \cup \{y_x\} \in \supp(\gamma)
\qand 
\rho \cup \{y_x\} \notin \Delta. 
$$
Now since $|\rho\setminus \{x\}|=k=a+b$, we can partition $\rho \setminus\{x\}$ into subsets $\sigma'$ and $\tau'$ (possibly empty) where
$$
\sigma' \cup \tau'=\rho \setminus \{x\}, 
\quad
\sigma' \cap \tau'=\emptyset,
\quad
|\sigma'|=a, \qand
|\tau'|=b.
$$
Then the two faces 
$$
\sigma=\sigma'\cup \{x\} \subseteq \rho
\qand 
\tau=\tau'\cup \{y_x\} \subseteq \big ( \rho \setminus \{x\} \big ) \cup \{y_x\}
$$
of $\supp(\gamma)$ satisfy 
$$ \dim(\sigma) = a, 
\quad 
\dim(\tau) = b, 
\quad 
\sigma \cap \tau =\emptyset,
\qand 
\sigma \cup \tau \notin \Delta 
$$
as desired.
\end{proof}

Using the lemma we can now prove \cref{thm:simplicialcomplex}: that a non-acyclic simplicial complex has complementary faces with complementary homologies  in the sense of \cref{que:main}.

\begin{proof}[Proof \cref{thm:simplicialcomplex}]
Since $\overline{L}$ is homeomorphic to $\Delta$, we have $\rhk_{k}(\Delta;\field) \neq 0$. Fix $a,b\geq -1$ with $a+b = k-2$. Therefore, by \Cref{p:almost-complements}, there exist $\sigma, \tau \in \Delta$ such that
$$
\dim(\sigma) = a+1, 
\quad 
\dim(\tau) = b+1, 
\quad 
\sigma \cap \tau = \emptyset, 
\qand
\sigma \cup\tau \notin  \Delta.  
$$
Since $\sigma \cap \tau = \emptyset$, we get $\sigma \wedge \tau = \hat{0}$ in $L$ and since $\sigma \cup\tau \notin  \Delta$, we get $\sigma \vee \tau =\hat{1}$ in $L$. Thus, $\tau\in \MC(\sigma)$. Since the face poset of a simplicial complex is a simplicial poset, the intervals $[\hat{0}, \sigma]$ and $[\hat{0}, \tau]$ are respectively Boolean lattices on sets of $a+2$ and $b+2$ elements. Therefore,
$$
(\hat{0}, \sigma)_L \iso S^{a} 
\qand 
(\hat{0}, \tau)_L \iso S^{b}.
$$
This completes the proof.
\end{proof}

%%%%%%%%%%%%%%%%%%%%%%%%%%%%%%%%%%
\section{The meet-lattice of a Cohen-Macaulay simplicial complex} \label{sec:meet}
%%%%%%%%%%%%%%%%%%%%%%%%%%%%%%%%%%

A simplicial complex is called \textit{Cohen-Macaulay} over a field $\KK$ if 
 for all $\sigma \in \Delta$ one has that $$\rhk_i(\lk_\Delta(\sigma);\field) = 0 \text{ 
 for all } i < \dim(\,\lk_\Delta(\sigma)\,).$$ This notion was coined by Stanley in \cite{Sta77} based 
 on work by Hochster \cite{Hoc77} and Reisner \cite{reisner76cmrings} showing that 
 Stanley's definition is equivalent to $\field[\Delta]$ being a Cohen-Macaulay ring. 
 Shortly after that, Quillen in \cite{Quillen78fibertherem} introduced a homotopy
 theoretic strengthening. A simplicial complex $\Delta$ is called \textit{homotopically
 Cohen-Macaulay} if for all $\sigma \in \Delta$ one has that $\lk_\Delta(\sigma)$ is
 homotopy equivalent to a (possibly empty) wedge of spheres of dimension 
 $\dim(\lk_\Delta(\sigma))$. Clearly, any homotopically Cohen-Macaulay complex is Cohen-Macaulay over any field.
 We call a poset $P$ Cohen-Macaulay over $\KK$ or
 homotopically Cohen-Macaulay if its order complex has the respective property.

This section is centered around the following question.

\begin{question}\label{que:CM} 
For what (homotopically) Cohen-Macaulay lattices $L$ is the meet lattice $L_\wedge $ also (homotopically) Cohen-Macaulay?
\end{question}

The following criterion is a simple consequence of the definitions above.

 \begin{lemma}
     Let $P$ be a bounded poset. 
     \begin{enumerate}
         \item $P$ is Cohen-Macaulay over $\KK$ if and only if for all $x \leq  y$ in $P$ 
         $$\rhk_i\big(\,(x,y)_P;\KK\,\big) = 0 \qfor i < \dim((x,y)_P)\,;
         $$
         \item $P$ is homotopically Cohen-Macaulay over $\KK$ if and only if 
          $(x,y)_P$ is homotopy equivalent to a wedge
           of spheres of dimension of the order complex of $(x,y)_P$.
     \end{enumerate}
 \end{lemma}

\cref{que:CM} holds for  lattices of small dimensions: if
$\dim(\overline{L}) = 0$ then $L = L_\wedge$; if $\dim(\overline{L}) = 1$
then Cohen-Macaulayness of $L$ is equivalent to $\overline{L}$ being
connected. The elements $x$ in $\overline{L} \setminus \overline{L_\wedge}$ are exactly those for which $(x,\hat{1})_L$ consists
of a single point. Their removal cannot disconnect $\overline{L}$. It can
reduce the dimension though.

We show below that in the case of the face lattice $L = \ML(\Delta)$ of a Cohen-Macaulay simplicial complex $\Delta$, \cref{que:CM} has a positive answer.
The next lemma prepares for the analysis of the Cohen-Macaulay property of $\ML(\Delta)_\cap$. 

\begin{lemma}\label{l:claim}
Let $\Delta$ be a simplicial complex which is Cohen-Macaulay over 
a field $\KK$ and let  $\tau \in \ML(\Delta)_\cap$ be a face not contained in every facet of $\Delta$. 
\begin{enumerate}
    \item If $(\hat{0}, \tau)_{\ML(\Delta)_\cap} = \emptyset$, then $\dim(\tau)=0$.
    \item If $(\hat{0}, \tau)_{\ML(\Delta)_\cap} \neq \emptyset$, then $\dim(\gamma) = \dim(\tau)-1$ for every maximal element $\gamma$ of $(\hat{0}, \tau)_{\ML(\Delta)_\cap}$.  
\end{enumerate}
\end{lemma}

\begin{proof}
(1) Since $(\hat{0}, \tau)_{\ML(\Delta)_\cap} = \emptyset$, for every facet $\gamma'$ of $\Delta$ not containing $\tau$, we have $\tau \cap \gamma' = \emptyset$. 
Fix such a facet $\gamma' \not \supseteq \tau$, and let $\gamma$ be a facet containing $\tau$.
From \cite[p. 177]{Bjo80} we know that a Cohen-Macaulay simplicial complex is facet-connected, also called gallery connected. 
Hence there is a sequence of facets of
$\Delta$, called a gallery, 
with 
$$
\gamma = \gamma_0, \gamma_1,\ldots, \gamma_k = \gamma' \qwhere \dim(\gamma_i \cap \gamma_{i+1}) = \dim(\Delta)-1
\qfor 0 \leq i \leq k-1\,.
$$
Now let $i$ be minimal such that
$\tau \subseteq \gamma_i$ and $\tau \not\subseteq \gamma_{i+1}$. Then, $\tau \cap \gamma_{i+1} = \emptyset$, so $\tau \subseteq \gamma_{i} \setminus \gamma_{i+1}$. But $\dim(\gamma_{i} \setminus \gamma_{i+1}) = 0$. It follows
that $\dim(\tau) = 0$. 

(2)
Let $\gamma$ be a maximal element of $(\hat{0}, \tau)_{\ML(\Delta)_\cap}$. 
Since all elements of $\overline{\ML(\Delta)_\cap}$ are non-empty faces of
$\Delta$, 
we have $\gamma \neq \emptyset$. 
The interval $(\gamma,\tau)_{\ML_\cap(\Delta)}$ is isomorphic to the interval
$(\hat{0},\tau \setminus \gamma)$ in 
$$
\ML(\,\lk_\Delta(\gamma) \,)_\cap=
\{ \rho \setminus \gamma \st \rho \in \ML(\Delta)_\cap \qand \gamma \subseteq \rho\}
.$$
Since $\gamma$ is a maximal element of $(\hat{0}, \tau)_{\ML(\Delta)_\cap}$, the interval $(\gamma,\tau)_{\ML(\Delta)_\cap}$ is empty; thus $(\hat{0},\tau \setminus \gamma)_{\ML(\lk_\Delta(\gamma))_\cap}$ is empty. 
Since the link of a Cohen-Macaulay complex is Cohen-Macaulay, $\lk_\Delta(\gamma)$ is Cohen-Macaulay. Therefore, by (1), $\dim(\tau\setminus \gamma')=0$. Hence, $\dim(\gamma) = \dim(\tau)-1$.
\end{proof}

\begin{theorem} Let $\Delta$ be a simplicial complex.
\begin{enumerate}
    \item
If $\Delta$ is Cohen-Macaulay over $\KK$
then $\ML(\Delta)_\cap$ is a Cohen-Macaulay lattice  over $\KK$.
\item If $\Delta$ is homotopically Cohen-Macaulay then $\ML(\Delta)_\cap$
is a homotopically Cohen-Macaulay lattice. 
\end{enumerate}
\end{theorem}

\begin{proof}
We prove (1). The proof of (2) follows along the same lines.

We need to check that every open interval $(\sigma,\tau)_{\ML(\Delta)_\cap}$
for $\sigma \leq \tau$ in $\ML(\Delta)_\cap$
has homology concentrated in dimension $\dim( (\sigma,\tau))$.
If $\sigma \neq \hat{0}$ then the interval $(\sigma,\tau)_{\ML(\Delta)_\cap}$ is isomorphic to the interval
$(\hat{0},\tau \setminus \sigma)_{\ML(\,\lk_\Delta(\sigma)\,)_\cap}$ in 
$$
\ML(\,\lk_\Delta(\sigma)\,)_\cap=
\{\, \rho \setminus \sigma \st \rho \in \ML(\Delta)_\cap \qand \sigma \subseteq \rho\,\}
.$$
Since the  link of a Cohen-Macaulay complex is Cohen-Macaulay, it suffices to consider the case 
$\sigma = \hat{0}$. Suppose $ \hat{0} \neq \tau \in \ML(\Delta)_\cap$.

If  $(\hat{0},\tau)_{\ML(\Delta)_\cap} =\emptyset$, or equivalently, $\tau$ is a minimal element of the lattice, then  $(\hat{0},\tau)_{\ML(\Delta)_\cap}$  has homology  concentrated in dimension $-1 = \dim(\,(\hat{0},\tau)_{\ML(\Delta)_\cap}\,)$, in which case we are done.

If $(\hat{0},\tau)_{\ML(\Delta)_\cap} \neq \emptyset$ and the intersection of all facets of $\Delta$ is non-empty, say $\pi$. Then, $\pi$ is the unique minimal element of $(\hat{0},\tau)_{\ML(\Delta)_\cap}$.  Hence, the homology  of $(\hat{0},\tau)_{\ML(\Delta)_\cap}$ vanishes in all dimensions and again we are done. 

It remains to consider the case where $(\hat{0},\tau)_{\ML(\Delta)_\cap} \neq \emptyset$ and the intersection of all facets of $\Delta$ is empty. Consider the simplicial complex 
$$
\Delta_\tau =\langle \gamma \subset \tau
\st 
\gamma \mbox{ maximal element in } (\hat{0},\tau)_{\ML(\Delta)_\cap}
\rangle \,.
$$
By \cref{l:claim} the facets of $\Delta_\tau$ are all of dimension $\dim(\tau)-1$, and therefore,  each two maximal elements $\gamma$ and $\gamma'$ differ by exactly one vertex, so any linear order on the facets of $\Gamma$ is a shelling order. Thus $\Delta_\tau$ is shellable, and hence homotopically Cohen-Macaulay.
Since every facet of $\Delta_\tau$ is the intersection of facets of $\Delta$ and
the facets of $\Delta_\tau$ are the maximal such intersections, it follows
that the interval $(\hat{0},\tau)_{\ML_\cap(\Delta)}$ is isomorphic to 
$\overline{\ML(\Delta_\tau)_\cap}$. 
In particular, since $\Delta_\tau$ is homotopically Cohen-Macaulay the homology of
$(\hat {0},\tau)_{\ML(\Delta_\tau)_\cap}$ is concentrated in dimension $\dim(\tau)-1$. 
\end{proof}

We now show that in general the answer to \cref{que:CM} is negative. Below is an example of a Cohen-Macaulay lattice $L$, where $L_\wedge$ is a non-Cohen-Macaulay  sublattice of $L$.
We refer the reader to \cite{JVcmposets} for the terminology used in the following example.

\begin{example}\label{ex:non-example}
    Let $L$ be a lattice on the set $\{0,\ldots,12\}$ that is the intersection of the permutations
        \begin{align*}
        &    \pi_1 = [0, 1,2,3,4,5,6,7,8,9,10,11,12]  \qand \\
        &    \pi_2 = [0,7,9,11,4,8,1,5,10,2,6,3,12]\; ;
    \end{align*}
    that is, $x \leq y$ in $L$ if and only if $x=y$ or $y$ is on right of $x$ in both $\pi_1$ and $\pi_2$.
    The lattice $L$ is shown in \cref{fig:L-vs-Lmeet} on the left. 
    Since $L$ is not a chain and it is intersection of two permutations, the dimension of $L$ is $2$.
    Furthermore, since $L$ is pure and the induced subposets of $L$ consisting of height $i$ and height $i + 1$ elements are connected for all $0 \leq i \leq \rank(L)-1$ (as one can see from \cref{fig:L-vs-Lmeet}), it follows that $L$ is Cohen-Macaulay \cite[Theorem 1]{JVcmposets}.
    
    Now the maximal elements of $\overline{L}$ are
    \[
    \facets{(L)} = \big\{s_1 = 3, \, s_2 = 6, \, s_3 = 10, \, s_4 = 11\big\}\,.
    \]
    We compute the elements of $\overline{L}$ which are the meet of maximal elements. Observe that
    \[
    2 = s_1\wedge s_2, \quad 5 = s_2\wedge s_3, \quad 9 = S_3\wedge S_4 \qand 1 = s_1\cap s_2\cap s_3,
    \]
    while the elements $4,7$ and $8$ can't be written as meet of maximal elements. The meet-lattice $L_\wedge$ is shown in \cref{fig:L-vs-Lmeet} on the right. One can see that $L_\wedge$ is not pure; indeed
    \[
     0 < 9 < 11 < 12 \quad \text{( colored blue)}  \qand 
        0 < 1 < 2 < 3 <12 \quad \text{(colored red)}
    \]
are two maximal chain of $L_\wedge$ length $3$ and $4$ respectively. Hence, it is not Cohen-Macaulay.
 
\end{example}

\begin{figure}[htbp]
  \centering
  \tikzset{
    hasse/.style={x=1.2cm, y=1.3cm, every node/.style={draw, rounded corners=2pt,
                  minimum width=6mm, minimum height=5mm, inner sep=1pt, font=\small}},
    keep/.style={},
    drop/.style={fill=gray!25, dashed},
    short/.style={fill=orange!20},
    edge/.style={thin},
    sedge/.style={thick, blue!80!black},
    redge/.style={thick, red!80!black},
  }

 \begin{tabular}{cc}
 \qquad \qquad 
    \begin{tikzpicture}[hasse]
      \node[keep] (top) at (2,4)  {$\scriptstyle{12}$};
      \node[keep] (A)   at (0.5,3) {$\scriptstyle{3}$};
      \node[keep] (B)   at (1.5,3) {$\scriptstyle{6}$};
      \node[keep] (D)   at (2.5,3) {$\scriptstyle{10}$};
      \node[keep] (C)   at (3.5,3) {$\scriptstyle{11}$};
      \node[keep] (r)   at (0.5,2) {$\scriptstyle{2}$};
      \node[keep] (s)   at (1.5,2) {$\scriptstyle{5}$};
      \node[keep] (t)   at (2.5,2) {$\scriptstyle{8}$};
      \node[keep] (u)   at (3.5,2) {$\scriptstyle{9}$};
      \node[keep] (b)   at (1,1)   {$\scriptstyle{1}$};
      \node[keep] (a)   at (2,1)   {$\scriptstyle{4}$};
      \node[keep] (c)   at (3,1)   {$\scriptstyle{7}$};
      \node[keep] (bot) at (2,0)   {$\scriptstyle{0}$};
      \draw[edge] (top) -- (A) (top) -- (B) (top) -- (D) (top) -- (C);
      \draw[edge] (A) -- (r) (B) -- (r) (B) -- (s) (D) -- (s) (D) -- (t) (D) -- (u) (C) -- (u);
      \draw[edge] (r) -- (b) (s) -- (b) (s) -- (a) (t) -- (a) (t) -- (c) (u) -- (c);
      \draw[edge] (b) -- (bot) (a) -- (bot) (c) -- (bot);
    \end{tikzpicture}
\qquad  \qquad & \qquad \qquad
    \begin{tikzpicture}[hasse]
      \node[keep] (top) at (2,4)  {$\scriptstyle{12}$};
      \node[keep] (A)   at (0.5,3) {$\scriptstyle{3}$};
      \node[keep] (B)   at (1.5,3) {$\scriptstyle{6}$};
      \node[keep] (D)   at (2.5,3) {$\scriptstyle{10}$};
      \node[keep] (C)   at (3.5,3) {$\scriptstyle{11}$};
      \node[keep] (r)   at (0.5,2) {$\scriptstyle{2}$};
      \node[keep] (s)   at (1.5,2) {$\scriptstyle{5}$};
      \node[keep] (u)   at (3.5,2) {$\scriptstyle{9}$};
      \node[keep] (b)   at (1,1)   {$\scriptstyle{1}$};
      \node[keep] (bot) at (2,0)   {$\scriptstyle{0}$};
      \draw[edge] (top) -- (B) (top) -- (D);
      \draw[edge] (B) -- (r) (B) -- (s) (D) -- (s) (D) -- (u);
      \draw[edge]  (s) -- (b);
      \draw[redge] (top) -- (A) -- (r) -- (b) -- (bot);
      \draw[sedge] (top) -- (C) -- (u) -- (bot);
    \end{tikzpicture} \qquad \qquad \\
    &\\
    The Cohen--Macaulay lattice $L$ 
    &
    The meet-lattice $L_\wedge$  
  \end{tabular}       
  \caption{The lattices $L$ and $L_\wedge$ in \cref{ex:non-example}}
  \label{fig:L-vs-Lmeet}
\end{figure}

%%%%%%%%%%%%%%%%%%%%%%%%%%%%%%%%%%%%%%%%%%
\section{Complementary Betti numbers of monomial ideals}\label{s:Betti}
%%%%%%%%%%%%%%%%%%%%%%%%%%%%%%%%%%%%%%%%%%

Statements about lattice  complementation and homology have direct consequences for  Betti numbers of monomial ideals.  
Let $S=\field[x_1,\ldots,x_n]$ be a polynomial ring in $n$ variables over a field $\field$, and $I=(\bm_1,\ldots,\bm_q)$ be a monomial ideal of $S$ generated by the monomials $\bm_1,\ldots,\bm_q$. A \textit{ minimal free resolution} of $S/I$ (see e.g., \cite{HHmonomialideals}) is the unique (up to isomorphism of complexes) exact sequence of free $S$-modules 
$$
0 \to S^{\beta_p} \to S^{\beta_{p-1}} \to \cdots \to S^{\beta_1} \stackrel{\partial}{\to} S
$$
where $\mbox{coker}(\partial)=S/I$. The positive integers $\beta_1,\ldots,\beta_p$ 
called the \textit{ Betti numbers} of $I$. Refining the minimal free resolution  into a \textit{ graded} or \textit{ multigraded} minimal free resolution leads to a refinement of the Betti numbers so that each Betti number can be written as a sum of graded or multigraded Betti numbers as follows:
$$
\beta_i(S/I)=\sum_{j \geq i} \beta_{i,j}(S/I) 
\qand 
\beta_{i,j}(S/I)=\sum_{\substack{\bm \in \LCM(I) \\ \deg(\bm)=j}} \beta_{i,\bm}(S/I)
\qfor j\geq i \geq 1.
$$

As indicated earlier, the multigraded Betti numbers of $S/I$ can be computed in terms of reduced homology of the order complex of open intervals in $\LCM(I)$~\cite[Theorem~2.1]{GPW99}:
\begin{equation}\label{eqn:Bettiformula} 
\beta_{i,\m}(S/I) = \dim_\field \ \rhk_{i-2}\big(\,(\hat{0},\m)_L;\field\,\big) \qfor \m\in L =\LCM(I)\,.
\end{equation}
  
Motivated by \cref{p:equivalences}, it is then natural to 
consider lattice complementation in the context of Betti numbers.

\begin{definition}[{\bf Complementary lcms}]\label{def:lcm-complement}
If $I$ is an ideal  of a polynomial ring minimally generated by monomials $\bm_1,\ldots,\bm_q$, let 
 $$
 L =\LCM(I) \qand
  \lcm(I)=\lcm(\bm_1,\ldots,\bm_q).
$$ 
Two monomials $\bm$ and $\bm'$ in $L$ are \textit{ complements} if 
\begin{enumerate}
\item $\lcm(\bm,\bm') =\lcm(I)$;
\item $\gcd(\bm,\bm') \notin I$.
\end{enumerate}
We denote the set of complements of $\bm$ by $\MC(\bm)$.
\end{definition}

Baclawski's~\cite{Bac77} result on lattice complementation then translates 
into the following statement: Suppose  $J$ is a monomial ideal of $S$ and $\beta_{i, \bm}(S/J) \neq 0$ for some $i>0$ and  $\bm \in \LCM(J)$.  Let $L=\LCM(I)$ where $I$ is generated by the minimal monomial generators $\bm_1,\ldots,\bm_q$ of $J$ which divide $\bm$. Then every $\bm' \in L$ has  a complement $\bm'' \in L$. In particular, 
$$\gcd(\bm',\bm'')\notin (\bm_1,\ldots,\bm_q) \qand
\lcm(\bm',\bm'')=\bm.
$$
This means that every multidegree strictly dividing that of a non-vanishing Betti number is complemented in the \say{induced} LCM-lattice.  Statement~($I$) in \cref{p:equivalences} then leads to the following definition, which takes into account complementary homologies.

\begin{definition}[{\bf Complementary Betti numbers}]  Let $I$ be a monomial ideal with non-vanishing multigraded Betti numbers 
$$\beta_{a,\m}(S/I) \neq 0 \qand \beta_{b,\m'}(S/I) \neq 0.
$$ Then  $\beta_{a,\m}$ and $\beta_{b,\m'}$ are called \textit{complementary} if 
\begin{enumerate}
    \item $\m' \in \MC(\m)$; 
    \item $\beta_{a+b,\lcm(I)} \neq 0$.
\end{enumerate}
If $\beta_{a,\m}$ and $\beta_{b,\m'}$ are complementary, we say $\beta_{a,\m}$ is \textit{ complemented} (similarly $\beta_{b,\m'}$). 
\end{definition}

The definition of  complementary multigraded Betti numbers  was inspired by the question of subadditivity of degrees of syzygies of monomial ideals. The relevant question in this context is the following (see~\cref{p:equivalences}~($I$) and~\cite[Question~1.1]{Faridilatticecomplements}).
\begin{question}\label{ques:Bettilcmlattice}
    If $I$ is a monomial ideal, 
    $\beta_{a+b,\lcm(I)}(S/I)\neq 0$ for some $a,b>0$, are there complementary 
    Betti numbers $\beta_{a,\m}(S/I)\neq 0$ and $\beta_{b,\m'}(S/I)\neq 0$?
\end{question}

We briefly describe how \cref{ques:Bettilcmlattice} relates to the subadditivity of degrees of syzygies.
For a positive integer $a$, define
\[
t_a(S/I) = \max\{j \st \beta_{a,j} \neq 0\}.
\]
The degrees of Betti numbers of $S/I$ satisfy the \textit{subadditivity property} if
$$
t_{a+b}(S/I) \leq t_a(S/I) +t_b(S/I) \qforall a,b >0 
$$
with $a+b$ bounded above by  the projective dimension of $S/I$. An affirmative answer to  \Cref{ques:Bettilcmlattice} would establish the subadditivity property for all monomial ideals, with the argument going as follows:  assume, for convenience,  that $I$ is a square-free monomial ideal with non-vanishing  Betti number $\beta_{i,x_1\cdots x_n}(S/I) \neq 0$, making $t_i(S/I)=n$. Then if $\bm$, $\bm'$, $a$, and $b$ are as in \cref{ques:Bettilcmlattice}, we will have
$$
t_a(S/I) + t_b(S/I)
\geq \deg(\m) +\deg(\m')
\geq n 
=t_i(S/I).
$$
In particular, this argument shows that condition~(1) in \cref{def:lcm-complement} is sufficient for settling the subadditivity question for monomial ideals. Indeed, a positive answer to \Cref{ques:Bettilcmlattice} and hence to the subadditivity question was given in \cite{synor24subadditivity} under the weaker assumption that $\m \vee \m' = \lcm(I)$, rather than requiring $\m$ and $\m'$ to be complements.  

While we are still not able to give a full answer to \cref{ques:Bettilcmlattice},  \cref{thm:complement} allows us to get much closer to one. 
A consequence of \cref{thm:complement} for multigraded Betti numbers can now be stated.

\begin{corollary}\label{c:Betti-complemented} If $I$ is a monomial ideal with $\beta_{i,\lcm(I)}(S/I)\neq 0$ for some $i\geq 2$, then for every $\m \in \LCM(I)\setminus \{\hat{0},\lcm(I)\}$ there are $\bn,\bm' \in \LCM(I) \setminus \{\hat{0}, \lcm(I)\}$ and complementary Betti numbers 
$$
\beta_{a,\bn}(S/I) \neq 0 \qand 
\beta_{b,\bm'}(S/I) \neq 0 
\qwhere a+b=i, \quad  a,b>0,
$$ 
and where  
$$
\bn \mid \m \qand 
\bm' \in \MC(\m) \cap \MC(\bn).
$$
In particular,
\[
t_{a+b}(S/I) \leq t_a(S/I) + t_b(S/I).
\]
\end{corollary}

\begin{proof}
Since $\beta_{i,\lcm(I)}(S/I)\neq 0$, it follows from \eqref{eqn:Bettiformula} that $\rhk_{i-2}(\overline{\LCM(I)}; \field) \neq 0$. 
 By \Cref{thm:complement}, for every
$\m\in \overline{\LCM(I)}$, there exist elements $\bm' \in \MC(\m)$, $\bn \in (\hat{0},\m]_{L}\cap \MC(\bm')$, together with integers
$a',b'\geq -1$ satisfying $a'+b'=i-2$,
such that
\[
\rhk_{a'}\big((\hat{0},\n)_L;\field\big)\neq 0
\qquad\text{and}\qquad
\rhk_{b'}\big((\hat{0},\m')_L;\field\big)\neq 0.
\]
The proof follows by setting $a=a'+2$ and $b=b'+2$ and using \eqref{eqn:Bettiformula}.
\end{proof}

The following statements are immediate consequences of \cref{thm:complement} and \cref{cor:1complemented}.
\begin{corollary}
 Let $I$ be a monomial ideal with $\beta_{i,\lcm(I)}(S/I)\neq 0$ and $\m \in \LCM(I)$. 
 \begin{enumerate}
     \item If $\m$ is a minimal element of $\MC(\m')$ for every $\m'\in \MC(\m)$ (or in particular, for $\m' \in \MC(\bm)$ chosen as in \Cref{thm:complement}), then $\beta_{a,\m} \neq 0$ for some $a$. 
     \item If $\m \in \MinGen(I)$, then there is $\m' \in \MC(\m)$ with 
     $\beta_{i-1,\m'}(S/I) \neq 0$.
\end{enumerate}
\end{corollary}

Besides the special cases that appeared in~\cite{Faridilatticecomplements,FM22breakinghomology}, \cref{c:Betti-complemented} is the best known answer to \cref{ques:Bettilcmlattice}. It is worth noting that not every multigraded Betti number should be expected to be complemented, as demonstrated by the example below.

\begin{example}[{\bf Not every Betti number is complemented}]\label{ex:noncomplemented}

    Let $S =\field[x_1,\ldots,x_{10}]$ and 
    \[
    I = (x_1x_2x_5, x_1x_4x_{10}, x_2x_4x_{10}, x_1x_5x_9, x_2x_5x_7, x_3x_7x_8, x_5x_6x_{10}, x_4x_5x_8).
    \]
    \cref{example:counter-example} contains a discussion on the LCM lattice of this ideal. Using  \eqref{eq:lcm2}, we observe that 
    $\beta_{i,x_1\cdots x_{10}} \neq 0$ only when $i=5$. Consider the element $\m = x_1x_2x_5x_7$ which has a unique complement 
     $\m^\perp = x_1x_3x_4x_5x_6x_7x_8x_9x_{10}$ in $\LCM(I)$ (see \cref{fig:interval}). However, $\beta_{i,\m}\neq 0$ only when $i=2$, and $\beta_{i,\m^\perp}\neq 0$ only when $i=4$. Therefore, $\beta_{2, \m}$ does not have a complementary Betti number. 
      \end{example}

Under the hypothesis of \Cref{c:Betti-complemented}, let $\m \in \LCM(I)\setminus \{\hat{0},\lcm(I)\}$ be such that $\beta_{a,\m}$ admits a complementary Betti number for some $a>0$, say $\beta_{b,\m'}$. One might then expect that, in \Cref{c:Betti-complemented}, the choices $\n = \m$ and $\bu = \m'$ would suffice. However, this need not be the case, as demonstrated in the following example. 

\begin{example}
    Let $S = \rationals[x_1,\ldots,x_{9}]$ and $$I = (x_1x_2x_{9}, x_1x_2x_6, x_2x_3x_8, x_6x_7x_8, x_1x_6x_8, x_1x_3x_8, x_3x_4x_5, x_6x_8x_{9}).$$
    Computations in Macaulay2~\cite{M2} show that $\beta_{i, x_1 \cdots x_9} \neq 0$ if and only if $i = 5$.
    Now consider the element $\m = x_1 x_2 x_6 x_8$.
    The set of complements of $\m$ is
\[
\MC(\m) = \{\m' = x_3 x_4 x_5 x_6 x_7 x_8 x_9, \m'' = x_2 x_3 x_4 x_5 x_6 x_7 x_8 x_9\}.
\]
Moreover, $\beta_{b,\m'} \neq 0$ if and only if $b = 3$, whereas $\beta_{c,\m''}\neq 0$ if and only if $c = 4$. Thus, $\beta_{2,\m}$ and $\beta_{3,\m'}$ are complementary Betti numbers. However, the interval $(\m', x_1 \cdots x_9) \setminus \MC(\m)$ is acyclic. For the element $\m$, the elements $\n$ and $\bu$ in \Cref{c:Betti-complemented} are $x_1x_2x_6$ and $\m''$, respectively.
 \end{example}

%%%%%%%%%%%%%%%%%%%%%%%%%%%%%%%%%%
\section{Open questions} \label{sec:discussion}
%%%%%%%%%%%%%%%%%%%%%%%%%%%%%%%%%%

As mentioned in the introduction Bj\"orner's result \cite[Theorem 3.3]{Bjo81}
postulates the existence of not necessarily distinct complements $y$ and $y'$  to a given $x$ such that $y$ is a join of atoms and $y'$ is a meet
of coatoms.
\cref{que:main} asks for pairs of complements $x$ and $y$ and strengthens the 
condition of being a
join of atoms to $(\hat{0},y)_L$ having non-trivial homology in a specific homological degree.
Strengthening the meet condition from \cite[Theorem 3.3]{Bjo81} in a similar way, one could ask the following
question.

\begin{question} \label{que:extmain-new}
Let $L$ be a lattice and suppose that
\begin{enumerate}
\item $\rhk_k\big(\,\overline{L};\KK\,\big) \neq 0$ for some $k \geq 0$,
\item $a,b \geq -1$ and $a+b = k-2$.
\end{enumerate}
Does there exist $x \in \overline{L}$ and complements $y$ and $y'$ of $x$ such that
\begin{align*}
&\rhk_a\big(\,(\hat{0},x);\KK\,\big) \neq 0, 
\qquad 
\rhk_b\big(\,(x,\hat{1});\KK\,\big) \neq 0, \\
&\rhk_b\big(\,(\hat{0},y);\KK \,\big) \neq 0,  
\qquad
\rhk_a\big(\,(y',\hat{1});\KK \,\big) \neq 0\,?
\end{align*}
\end{question}

\cref{p:geometric} provides a positive answer to \cref{que:extmain-new} for geometric lattices of rank $k+2$. Moreover, in this case  one could choose any $x$ and one could achieve $y = y'$.
Recall that a lattice $L$ is called a {\it geometric lattice} (see \cite[pp.~4]{Rot64} or \cite{Bjo92}) if every element of $L$ is the join of
atoms, and whenever $x$ and $y$ in $L$ cover $x\wedge y$, then $x\vee y$ covers both $x$ and $y$.  It is known that a geometric lattice of rank $k+2$ is non-acyclic~\cite[Theorem 4(a)]{Rot64} with homology concentrated in dimension $k$ (\cite[Theorem 4.1]{Fol66}).

It follows that  the answer to \cref{que:extmain-new} is \say{yes} for
geometric lattices $L$. Condition (1) of \cref{que:extmain-new} implies that
$L$ must be of rank $k+2$. 

\begin{proposition}\label{p:geometric}
For integers  $a,b \geq -1$, let $L$ be a geometric lattice of rank $a+b+4$. Then for every $x \in \overline{L}$ of rank $a+2$, there exists $y \in \MC(x)$ (of rank  $b+2$) such that
\begin{align*}
&\rhk_a\big(\,(\hat{0},x);\KK\,\big) \neq 0, 
\qquad 
\rhk_b\big(\,(x,\hat{1});\KK\,\big) \neq 0, \\
&\rhk_b\big(\,(\hat{0},y);\KK \,\big) \neq 0,  
\qquad
\rhk_a\big(\,(y,\hat{1});\KK \,\big) \neq 0\,.
\end{align*} 
\end{proposition}

\begin{proof}
Set $k= a+b+2$. Since $L$ is a geometric lattice 
of rank $a+b+4=k+2$ we have $\rhk_k\big(\,\overline{L};\KK\,\big) \neq 0$. 
As indicated in \cite{Bjo92}, $L$ is the lattice of flats of a matroid $M$ with $\rk(M)=\rk(L)$, where  $\rk$ is the rank function of $L$ and $M$. 
For a flat $X$ of rank $a+2 \geq 1$ in $L$ choose a basis
$e_1,\ldots, e_{a+2}$ of $X$ and extend this basis to a basis $e_1,\ldots, e_{k+2}$ of $M$. Let 
$Y$ be the flat with basis $e_{a+3},\ldots, e_{k+2}$. Then $Y$ is of rank $(k+2)-(a+2) = k-a = b+2$.  
By construction
$X \vee Y = M = \hat{1}$ in $L$. On the other hand, as $L$ is a geometric lattice, we have  
$$
\begin{array}{ccccccccccc}
&&\rk(X)&+&\rk(Y)&\geq&\rk(X \vee Y)&+&\rk(X\wedge Y)&&\\
&&\rotatebox{90}{=}&&\rotatebox{90}{=}&&\rotatebox{90}{=}&&\rotatebox{90}{=}&&\\
k+2&=&(a+2)&+&(b+2)&\geq&(k+2)&+&\rk(X\wedge Y)&&\\
\end{array}
$$
It follows that $X \wedge Y = \hat{0}$. 
In the geometric lattice $L$ for any flat $Z$ the interval $[\hat{0},Z]$ is a geometric lattice of
rank $\rk(Z)$ and the interval $[Z,\hat{1}]$ is a geometric lattice of rank $(k+2)-\rk(Z)$ (see \cite[\P 3]{Bjo92}), hence $\rhk_{k-\rk(Z)}\big(\,(Z,\hat{1});\KK \,\big) \neq 0$.

In the case $Z=X$ we have $\rk(X)=a+2$ and $(k+2)-\rk(X)=b+2$, implying that
$$
\rhk_a\big(\,(\hat{0},x);\KK\,\big) \neq 0
\qand 
\rhk_b\big(\,(x,\hat{1});\KK\,\big) \neq 0\,.
$$
The case $Z=Y$ is similar. This completes the proof.
\end{proof}

The next class of examples is a class for which the positive answer to
\cref{que:extmain-new} hinges on the existence of an element $x$ with the
desired homological properties. Again, if this element exists then 
one can achieve $y = y'$.

\begin{example} \label{ex:ortho}
    A lattice $L$ is called \textit{ortho-complemented} (see e.g. \cite[Chapter 3.2]{BB05}) 
    if there is a map $\perp : L \rightarrow L$ satisfying for all $x,y \in L$:
    \begin{itemize}
        \item $x \vee x^ \perp = \hat{1}$ and $x \wedge x^ \perp = \hat{0}$,
        \item $x \leq y$ $\iff$ $y^ \perp \leq x^ \perp$,
        \item $(x^ \perp)^ \perp = x$.
    \end{itemize}
    It is an immediate consequence of the axioms that for all $x \in L$ the interval
    $(\hat{0},x)_L$ is isomorphic to the order dual of $(x^\perp,\hat{1})$ and $(x,\hat{1})_L$ is
    isomorphic to the order dual of $(\hat{0},x^\perp)_L$.
    From this the positive answer to \cref{que:extmain-new} follows for $y = x^ \perp$
    for an ortho-complemented lattice $L$ if and only if  
    from $\rhk_{k}(\,\overline{L};\KK\,) \neq 0$ the existence of
    an $x \in L$ with $\rhk_a \big( \,(\hat{0},x)_L;\KK\,\big), \rhk_b \big(\,(x,\hat{1})_L;\KK\,\big)\neq 0$ can be deduced. Thus any class of ortho-complemented lattice satisfying this
    assumption gives rise to a positive answer to \cref{que:extmain-new}.
    \end{example}

The following example shows that there are large classes of ortho-complemented lattices
which fail to have $x$ as demanded in the preceding example. 

\begin{example}
    Let us consider the weak Bruhat order $L$ for a finite Coxeter group $W$ with 
    Coxeter generating system $S$. By \cite[Corollary 3.2.2]{BB05} ortho-complementation is given
    by $x^\perp = x w_S$. Here for any $S' \subseteq S$ we let $w_{S'}$ be the longest word in 
    $S'$. By \cite[Theorem 3.2.7]{BB05} it follows
    that $\rhk_k(\,\overline{L};\KK\,) \neq 0$ if and only if $k = |S|-2$. Now let $a,b$ be integers 
    such that $a+b = k-2$.
    Then the same theorem implies that
    $\rhk_a\big(\,(\hat{0},x)_L;\KK\,\big) \neq 0$ if and only if $x$ is the longest word in some subset
    $S' \subseteq S$ of size $a+2$ and $0$ for
    all other $x$. This also shows that if $y$ is a complement to $x$ with 
    $\rhk_b\big(\,(\hat{0},y)_L;\KK\,\big) \neq 0$ then $y$ is the longest word in 
    some subset $S'' \subseteq S$ of size $b+2$. By $x \wedge y = \hat{0}$ this implies
    $S' \cap S'' = \emptyset$. Now $a+b = k-2 = |S|-2-2$ implies that $S' \cup S'' = S$.
    Conversely, for any partition of $S$ into an $a+2$ set $S'$ and a $b+2$ element set $S''$
    the longest word $x$ on $S'$ and the longest word $y$ on $S''$ are complements.
    This in particular, shows that the answer to \cref{que:main} is positive for the
    weak Bruhat order. 
    Now consider the interval $(x,\hat{1})_L$ for $x$ the longest word on 
    an $a+2$ element subset $S'$ of $S$. In order to have 
    $\rhk_b\big(\,(x,\hat{1})_L;\KK\,\big) \neq 0$ again \cite[Theorem 3.2.7]{BB05}
    we must have that there is a subset $S'' \subseteq S$ of size $b+2$ such that
    for the longest word $y$ in $S''$ we have that $xy$ is the longest word in $S$. 
    By $a+b = k-2 = |S|-4$ and the fact that every element of $S$ must appear in $xy$
    it follows again that $S' \cap S'' = \emptyset$ and $S' \cup S'' = S$.
    Standard arguments show that this can only be the case if the Coxeter diagram of $S$
    is disconnected and the diagrams of $S'$ and $S''$ are unions of some of its connected
    components. In particular, it never happens if $W$ is an irreducible Coxeter group.
 \end{example}
  
Note that in the counterexample above, the dimension of the order complex 
of the lattice is strictly larger than the homological dimension in which non-vanishing homology occurs.
Indeed, in \cref{p:geometric} the number $k$ is the dimension of the order complex
of $\overline{L}$.

The following statement implies that if $k$ is the dimension of the order complex of $\overline{L}$ 
then we always find $x$ satisfying the conditions of \cref{que:extmain-new}.

\begin{proposition} \label{lem:updown}
Let $L$ be a lattice such that the order complex of $\overline{L}$ has
dimension $k=a+b+2$ for some $a,b \geq -1$
and let $\gamma$ be a $k$-cycle of the order complex.

\begin{enumerate} 

\item If $x_0 < \cdots < x_k$ is a chain in the 
support of $\gamma$ then 
$$\rhk_{a} \big(\,(\hat{0},x_{a+1})_L;\KK\,\big) \neq 0 \text{ and } 
\rhk_{b} \big(\,(x_{a+1},\hat{1})_L;\KK\,\big) \neq 0\,.$$

\item If  there are chains
$$
x_0 < x_1 < \cdots < x_k 
\qand 
y_0 < y_1 < \cdots < y_k
$$ 
in the support of $\gamma$ such that 
$x_{a+1}$ and $y_{b+1}$ are complements, then 
\begin{align*}
&\rhk_a\big(\,(\hat{0},x_{a+1});\KK\,\big) \neq 0, 
\qquad 
\rhk_b\big(\,(x_{a+1},\hat{1});\KK\,\big) \neq 0, \\
&\rhk_b\big(\,(\hat{0},y_{b+1});\KK \,\big) \neq 0,  
\qquad
\rhk_a\big(\,(y_{b+1},\hat{1});\KK \,\big) \neq 0\,.
\end{align*} 
\end{enumerate}
\end{proposition}

\begin{proof}
 The proof of~(1) follows immediately from the definition of the simplicial differential and is left to
the reader. Statement~(2) is an immediate consequence of~(1) applied to $y_{b+1}$ and with the
    roles of $a$ and $b$ interchanged.
\end{proof}

\cref{lem:updown} leads to the following question, to which we have no counterexample. 

\begin{question} 
    Is the answer to \cref{que:extmain-new} positive if $k$ coincides with the 
    dimension of the order complex of $\overline{L}$?
\end{question}

%\bibliography{./ref.bib}{}
%\bibliographystyle{alphaurl}

\newcommand{\etalchar}[1]{$^{#1}$}

\end{document}